\documentclass{amsart}
\usepackage[margin=1in]{geometry}
\usepackage{amsmath}
\usepackage{amssymb}
\usepackage{amsthm}
\usepackage{amscd}
\usepackage{enumerate}

\newcommand{\EE}{\Bbb E}

\newcommand{\ZZ}{\Bbb Z}

\newcommand{\NN}{\Bbb N}

\newcommand{\CC}{\Bbb C}
\newcommand{\ip}[1]{\langle #1 \rangle}

\newcommand{\widetidle}{\widetilde}

\newcommand{\varpesilon}{\varepsilon}

\newcommand{\varespilon}{\varepsilon}

\newcommand{\actson}{\curvearrowright}

\newtheorem{question}{Question}

\newtheorem{?}{Question}
\newtheorem{theorem}{Theorem}
\newtheorem{definition}[theorem]{Definition}
\newtheorem{proposition}[theorem]{Proposition}
\newtheorem{cor}[theorem]{Corollary}
\newtheorem{lemma}[theorem]{Lemma}

\DeclareMathOperator{\Int}{int}

\DeclareMathOperator{\Span}{Span}

\DeclareMathOperator{\Map}{Map}

\DeclareMathOperator{\id}{Id}

\DeclareMathOperator{\tr}{tr}

\DeclareMathOperator{\Hom}{Hom}

\DeclareMathOperator{\alg}{alg}

\DeclareMathOperator{\CO}{CO}
\DeclareMathOperator{\AP}{AP}
\DeclareMathOperator{\Prob}{Prob}
\DeclareMathOperator{\Ball}{Ball}

\numberwithin{theorem}{section}

\begin{document}
\title{Mixing and Spectral Gap Relative to Pinsker Factors for Sofic Groups}      % Enter your title between curly braces
\author{Ben Hayes}
\address{Stevenson Center\\
         Nashville, TN 37240}
\email{benjamin.r.hayes@vanderbilt.edu}
\date{\today}

\begin{abstract}

	Motivated by our previous results, we investigate structural properties of probability measure-preserving actions of sofic groups relative to their Pinsker factor. We also consider the same properties relative to the Outer Pinsker factor, which is another generalization of the Pinsker factor in the nonamenable case. The Outer Pinsker factor is motivated by extension entropy for extensions, which fixes some of the ``pathological'' behavior of sofic entropy: namely increase of entropy under factor maps. We show that an arbitrary probability measure-preserving action of a sofic group is mixing relative to its Pinsker and Outer Pinsker factors and, if the group is nonamenable, it has spectral gap relative to its Pinsker and Outer Pinsker factors. Our methods are similar to those we developed in ``Polish models and sofic entropy'' and based on representation-theoretic techniques. One crucial difference is that instead of considering unitary representations of a group $\Gamma$, we must consider $*$-representations of algebraic crossed products of $L^{\infty}$ spaces by $\Gamma.$

\end{abstract}

\maketitle
\tableofcontents

\section{Introduction}

	The goal of this paper is to further the investigation set out in \cite{Me6}  exploring the connections between representation theory and entropy of probability measure-preserving actions of groups, particularly for nonamenable groups. Entropy for probability measure-preserving actions of $\ZZ$ is classical and goes back to the work of Kolmogorov and Sina\v\i. Entropy is roughly a measurement of how ``chaotic'' the action of $\ZZ$ is. It was realized by Kieffer in \cite{Kieff} that one could define entropy for actions of amenable groups instead of $\ZZ$. Roughly, a group is amenable if it has a sequence of approximately translation-invariant non-empty finite subsets. Entropy theory for amenable groups has been well-studied and parallels the case of the integers quite well (see e.g. the seminal \cite{OrnWeiss} which reproves some of the fundamental isomorphism results in the amenable case).
	
	Fundamental examples in \cite{OrnWeiss} led many to believe that it was not possible to define entropy in a reasonable way for actions of nonamenable groups. In stunning and landmark work Bowen in \cite{Bow} developed a reasonable notion of entropy for the class of \emph{sofic} groups. Sofic groups are a class of groups vastly larger than amenable groups: they contain all amenable groups, all residually finite groups, all linear groups and are closed under free products with amalgamation over amenable subgroups (see \cite{ESZ2},\cite{DKP},\cite{LPaun},\cite{PoppArg}). Roughly a group is sofic if it has ``almost actions'' which are ``almost free'' on finite sets. Sofic entropy of a probability measure-preserving action $\Gamma\actson (X,\mu)$ then measures the exponential growth of the number of ``finitary simulations'' there are of the space which are compatible with the sofic approximation.
	
		We remark here that defining entropy for nonamenable groups is not merely generalization for generalizations' sake: the application of results in orbit equivalence and von Neumann algebras require showing that actions of nonamenable groups are not isomorphic. For example, in Bowen's first paper on the subject, he was able to use fundamental work of Popa in \cite{PopaStrongRigidity},\cite{PopaSG}, to show that if $\Gamma$ is an infinite conjugacy class, property (T), sofic group (e.g. $PSL_{n}(\ZZ)$ for $n\geq 3$) and $(X,\mu),(Y,\nu)$ are standard probability spaces with $H(X,\mu)\ne H(Y,\nu),$ then $L^{\infty}((X,\mu)^{\Gamma}\rtimes\Gamma\not\cong L^{\infty}((Y,\nu)^{\Gamma})\rtimes \Gamma$ (here $L^{\infty}((X,\mu)^{\Gamma})\rtimes \Gamma$ is the von Neumann algebra crossed product: a natural von Neumann algebra associated to any probability measure-preserving action). Bowen also gave similar applications to orbit equivalence rigidity. The use of sofic entropy is completely unavoidable for this result: it is known that if $\Gamma$ is a sofic group which contains $\ZZ$ as a subgroup, then sofic entropy is a \emph{complete} invariant for isomorphisms of Bernoulli shifts. Thus one cannot deduce nonisomorphism of crossed product von Neumann algebras or failure of orbit equivalence for Bernoulli actions of such groups without using sofic entropy. Of course one cannot prove such rigidity for actions of amenable groups as the crossed product von Neumann algebras they produce are always the same, by Connes' Theorem. Similar remarks apply to orbit equivalence by work of Ornstein-Weiss and Connes-Feldman-Weiss (see \cite{OrnWeiss}, \cite{CFW}).
	
	In \cite{Me6}, we expanded on connections to orbit equivalence theory. To summarize the results we need some terminology. If $\Gamma\actson (X,\mu)$ is a probability measure-preserving action, a \emph{factor} of the action is another probability measure-preserving action $\Gamma\actson (Y,\nu)$ so that there is an almost everywhere $\Gamma$-equivariant, measurable map $\pi\colon X\to Y$ so that $\pi_{*}\mu=\nu.$ We call $\pi$ a factor map. We sometimes say that $(X,\mu)\to (Y,\nu)$ is an extension and if we wish to specify the group we will say that $\Gamma\actson (X,\mu)\to \Gamma\actson (Y,\nu)$ is an extension. An action of a sofic group is said to have completely positive entropy if every nontrivial (i.e. not a one-point space) factor  has positive entropy (see \cite{KerrCPE},\cite{BurCPE} for examples of completely positive entropy actions of sofic groups). Lastly, a probability measure-preserving action $\Gamma\actson (X,\mu)$ is said to be strongly ergodic if for every sequence $A_{n}$ of measurable subsets of $X$ with $\mu(g A_{n}\Delta A_{n})\to 0$ for all $g\in\Gamma$ we have $\mu(A_{n})(1-\mu(A_{n}))\to 0.$ In Corollary 1.2 of \cite{Me6} we showed that every probability measure-preserving action of a nonamenable, sofic group with positive entropy is strongly ergodic. We remark that strong ergodicity is an invariant of the orbit equivalence class of the action. Thus, a particular consequence of our results is that if a probability measure-preserving action of a nonamenable group is not strongly ergodic, then no action orbit equivalent to it has completely positive entropy. This results stands in stark contrast to the celebrated fact that all ergodic, probability measure-preserving actions of amenable groups are orbit equivalent. The applications sofic entropy has to von Neumann algebra and orbit equivalence rigidity makes it clear that generalizing entropy to the nonamenable realm is a useful endeavor, as such rigidity phenomena \emph{never} occurs for actions of amenable groups.
	
	In this note, we expand on some of the results in \cite{Me6}. Because our results only apply for completely positive entropy actions, and there are few known examples of such actions, we wish to generalize our results in \cite{Me6} so that they give structural properties for arbitrary actions. In the amenable case it is well known how to do this: given any probability measure-preserving action of an amenable group there is a maximal factor, called the Pinsker factor, which has entropy zero. We can thus say that any action has completely positive entropy relative to its Pinsker factor and much of what is known for completely positive entropy actions is known for a general action ``relative to the Pinsker factor.'' For example, any action of an amenable group is mixing relative to its Pinsker factor. One can prove in the sofic case that there is a unique largest entropy zero factor of any action. We can also call this the \emph{Pinsker factor}. However, because entropy can decrease under factors this factor does not, in our opinion, have the right monotonicity properties and so we wish to also investigate another generalization of the Pinsker factor, called the \emph{Outer Pinsker factor}.
	
	Motivating the definition of the Outer Pinsker factor is a way to fix the ``pathological'' behavior that entropy can increase under factor maps. Implicit in an alternate formulation of entropy due to Kerr in \cite{KerrPartition}, given a countable, discrete, sofic group $\Gamma,$ a probability measure-preserving action $\Gamma\actson (X,\mu)$ and a factor $\Gamma\actson (Y,\nu)$ of $\Gamma\actson (X,\mu)$ we can define the ``extension entropy'' $h_{\mu}(Y:X,\Gamma).$ The extension entropy measures how many ``finitary simulations'' of $\Gamma\actson (Y,\nu)$ there are which ``lift'' to ``finitary simulations'' of $\Gamma\actson (X,\mu).$ It is easy to see that this has the right monotonicity properties: $h(Y:X,\Gamma)$ is decreasing under factors of $Y$ if we $X$ fixed, it is increasing under intermediate factors between $X$ and $Y$ if we keep $Y$ fixed, and it is subadditive under joins of factors in the first variable (it can be shown by methods analogous to \cite{KLi2} that $h_{(\sigma_{i})_{i},\mu}(Y:X,\Gamma)=h_{(\sigma_{i})_{i},\nu}(Y,\Gamma)$ if $\Gamma$ is amenable). Because of these monotonicity properties there is a canonical maximal factor $\Gamma\actson (Y,\nu)$, called the \emph{Outer Pinsker Factor}, so that the extension entropy $h(Y:X,\Gamma)$ is zero. We do not claim any originality on the definition of the Outer Pinsker factor, as its definition is quite natural (and appears to be folklore) and through private communication we know it has been observed at least  by L.Bowen, Kerr, Seward (each independent of the other). The goal of this note is merely to extend what representation-theoretic properties we know for completely positive entropy actions to spectral properties of an \emph{arbitrary} probability measure-preserving action relative to the Pinsker and Outer Pinsker factors. In the appendix, we observe that when $\Gamma$ is amenable we have $h_{\mu}(Y:X,\Gamma)=h_{\nu}(Y,\Gamma).$ The proof is essentially a combination of \cite{BowenAmen},\cite{KLi2} and Corollary 5.2 of \cite{PoppArg} (which was first proved under the assumption that the action is free in Proposition 1.20 of \cite{LPaun}).
	
	For entropy for single actions, instead of extension entropy, to deduced the desired spectral properties from assumptions of positive entropy it was enough to use just the unitary representation theory of the group. It turns out that in order to deduce our desired results for extensions
	\[\Gamma\actson (X,\mu)\to \Gamma\actson (Y,\nu),\]
we will need to know how both how $\Gamma$ and how $Y$ (or more precisely $L^{\infty}(Y)$) ``acts'' on $X.$ The right way to do this is to replace unitary representations of $\Gamma$ with $*$-representations of the algebraic crossed product: $L^{\infty}(Y)\rtimes_{\textnormal{alg}}\Gamma.$ Recall that the algebraic crossed product is the algebra of all finite formal sums:
\[\sum_{g\in\Gamma}f_{g}u_{g},\mbox{ $f_{g}\in L^{\infty}(Y)$ and all but finitely many $f_{g}$ are zero},\]
with the imposed relation
\[u_{g}f=(f\circ g^{-1})u_{g},\mbox{ $g\in\Gamma,f\in L^{\infty}(Y)$.}\]
Defining
\[\left(\sum_{g\in\Gamma}f_{g}u_{g}\right)^{*}=\sum_{g\in\Gamma}(\overline{f_{g}}\circ g)u_{g},\]
the algebraic crossed product becomes a $*$-algebra.  If
\[\pi\colon X\to Y\]
is a factor map, we have a $*$-representation $\rho$ of $L^{\infty}(Y)\rtimes_{\textnormal{alg}}\Gamma$ on $L^{2}(X)$ given by:
\begin{align*}
&(\rho(u_{g})\xi)(x)=\xi(g^{-1}x),\mbox{ for $g\in\Gamma$, $\xi\in L^{2}(X,\mu)$,}\\
&(\rho(f)\xi)(x)=f(\pi(x))\xi(x),\mbox{ for $f\in L^{\infty}(Y,\nu)$, $\xi\in L^{2}(X,\mu).$}
\end{align*}
In order to properly formulate our generalization of Theorem 1.1 in \cite{Me6}, we will need a $*$-representation of $L^{\infty}(Y,\nu)\rtimes_{\textnormal{alg}}\Gamma$ which may be regarded as an analogue of the left regular representation of a group. Von Neumann algebra theory provides us with the correct analogue: consider the $*$-representation $\lambda\colon L^{\infty}(Y,\nu)\rtimes_{\textnormal{alg}}\Gamma\to B(L^{2}(Y,\nu,\ell^{2}(\Gamma)))$ uniquely determined by
\begin{align*}
&(\lambda(f)\xi)(y)=f(y)\xi(y)\mbox{ for $y\in Y,$ $f\in L^{\infty}(Y),\xi\in L^{2}(Y,\nu,\ell^{2}(\Gamma)$},\\
&(\lambda(u_{g})\xi)(y)(h)=\xi(g^{-1}y)(g^{-1}h)\mbox{ for $y\in Y,g,h\in\Gamma$}.
\end{align*}
We will see that this is the correct analogue of the left regular representation. We note this $*$-representation of $L^{\infty}(Y)\rtimes_{\textnormal{alg}}\Gamma$ is precisely the one obtained from the action of the von Neumann algebra crossed product $L^{\infty}(Y)\rtimes \Gamma$ on its $L^{2}$-space. Recall that if $\rho_{j}\colon A\to B(\mathcal{H}_{j}),j=1,2$ are two $*$-representations of a $*$-algebra on Hilbert spaces $\mathcal{H}_{j},j=1,2$ then $\rho_{1},\rho_{2}$ are \emph{singular} written $\rho_{1}\perp\rho_{2}$ if and only if no subrepresentation of $\rho_{1}$ is embeddable into $\rho_{2}.$ Suppose we are given probability measure-preserving actions $\Gamma\actson (X,\mu),\Gamma\actson (Y,\nu),\Gamma\actson (Z,\zeta)$  of $\Gamma$ and  factor maps $\pi\colon X\to Y,\rho\colon X\to Z.$ We say that another factor $\Gamma\actson (Z,\zeta)$ is an intermediate factor between $X$ and $Y$ if there is a factor map $\phi\colon Z\to Y$ so that $\pi=\phi\circ \rho.$  We say that a set $\mathcal{F}$ of measurable functions $X\to\CC$ generates $Z$ if the smallest, complete, $\Gamma$-invariant sigma algebra of sets containing
\[\{f^{-1}(A):f\in\mathcal{F},A\subseteq \CC\mbox{ is Borel}\}\]
is
\[\{\rho^{-1}(E):E\subseteq Z\mbox{ is $\zeta$-measurable}\}.\]
Note that since we take complete sigma-algebras, this does not depend upon the elements of $\mathcal{F}$ mod null sets, so we can make sense of what it means for a subset of $L^{2}(X,\mu)$ to generate a factor. We are now ready to state the following analogue of the main theorem of \cite{Me6} for extension entropy.

\begin{theorem}\label{T:IntroMain} Let $\Gamma$ be a countable, discrete, sofic group with sofic approximation $\sigma_{i}\colon\Gamma\to S_{d_{i}}.$ Let $\Gamma\actson (X,\mu)$ be a measure-preserving action of $\Gamma$ with $(X,\mu)$ a standard probability space. Let $\Gamma\actson (Y,\nu)$ be a factor of $\Gamma\actson (X,\mu)$ and let $\Gamma\actson (Z,\zeta)$ be a intermediate factor in between $X$ and $Y.$ Suppose there exists a $L^{\infty}(Y)\rtimes_{\textnormal{alg}}\Gamma$-subrepresentation $\mathcal{H}$ of $L^{2}(X,\mu)$ which generates $Z$ and so that $\mathcal{H}$ is singular with respect to $L^{2}(Y,\nu,\ell^{2}(\Gamma))$ as a representation of $L^{\infty}(Y,\nu)\rtimes_{\textnormal{alg}}\Gamma.$ Then,
\[h_{(\sigma_{i})_{i},\mu}(Z:X,\Gamma)=h_{(\sigma_{i})_{i},\mu}(Y:X,\Gamma).\]
In particular,
\[h_{(\sigma_{i})_{i},\zeta}(Z,\Gamma)\leq h_{(\sigma_{i})_{i},\nu}(Y,\Gamma).\]

\end{theorem}

As in \cite{Me6}, the following description of  $L^{2}(X,\mu)$ as a representation of its Pinsker factor crossed product is automatic. Suppose $\pi\colon X\to Y$ is a factor map. Then we can view $L^{2}(Y)$ as a subspace of $L^{2}(X)$ via the embedding $f\mapsto f\circ \pi$ for $f\in L^{2}(Y).$

\begin{cor}\label{C:IntroRepstuff} Let $\Gamma$ be a countable, discrete, sofic group with sofic approximation $\sigma_{i}\colon\Gamma\to S_{d_{i}}.$ Let $\Gamma\actson (X,\mu)$ be an arbitrary measure-preserving action where $(X,\mu)$ is a standard probability space. Let $\Gamma\actson (Y_{0},\nu_{0}),\Gamma\actson (Y,\nu)$ be the Outer Pinsker factor and Pinsker factor of $\Gamma\actson (X,\mu)$ respectively. Then, as a representation of $L^{\infty}(Y)\rtimes_{\textnormal{alg}}\Gamma,$ we have that $L^{2}(X)\ominus L^{2}(Y)$ embeds into $L^{2}(Y,\ell^{2}(\Gamma))^{\oplus \infty}.$ Similarly, as a representation of $L^{\infty}(Y_{0})\rtimes_{\textnormal{alg}}\Gamma,$ we have that $L^{2}(X)\ominus L^{2}(Y)$ embeds into $L^{2}(Y_{0},\ell^{2}(\Gamma))^{\oplus \infty}.$

\end{cor}
	
Since this formulation in terms of algebraic crossed product is somewhat abstract and far from the ergodic theoretic roots of sofic entropy, we mention a purely ergodic theory corollary of Theorem \ref{T:IntroMain}. We say that an extension
\[\Gamma\actson (X,\mu)\to \Gamma\actson (Y,\nu)\]
is mixing if for all $\xi,\eta\in L^{\infty}(X)$ with $\EE_{Y}(\xi)=0=\EE_{Y}(\eta)$ we have
\[\lim_{g\to\infty}\|\EE_{Y}((\xi\circ g^{-1})\eta)\|_{L^{2}(Y)}=0.\]
Here $\EE_{Y}(f)$ is the conditional expectation of $f\in L^{1}(X,\mu)$ onto $Y.$ The extension is said to have spectral gap if for every sequence $\xi_{n}\in L^{2}(X)$ with
\[\|\xi_{n}\circ g^{-1}-\xi_{n}\|_{2}\to_{n\to\infty}0\mbox{ for all $g\in\Gamma$,}\]
we have
\[\|\xi_{n}-\EE_{Y}(\xi_{n})\|_{2}\to 0.\]
To make sense of $\xi_{n}-\EE_{Y}(\xi_{n})$ we are using the embedding of $L^{2}(Y)$ into $L^{2}(X)$ defined before via the factor map.

\begin{cor} Let $\Gamma$ be a countable, discrete, sofic group with sofic approximation $\sigma_{i}\colon\Gamma\to S_{d_{i}}.$ Let $\Gamma\actson (X,\mu)$ be an arbitrary measure-preserving action where $(X,\mu)$ is a standard probability space. Let $\Gamma\actson (Y_{0},\nu_{0}),\Gamma\actson (Y,\nu)$ be the Pinsker factor and Outer Pinsker factor of $\Gamma\actson (X,\mu)$ respectively.

(i): If $\Gamma$ is infinite, then $\Gamma\actson (X,\mu)$ is mixing relative to $\Gamma\actson (Y_{0},\nu_{0}).$ In particular, $\Gamma\actson (X,\mu)$ is mixing relative to $\Gamma\actson (Y,\nu).$

(ii) If $\Lambda$ is any nonamenable subgroup of $\Gamma,$ then $\Lambda\actson (X,\mu)$ has spectral gap relative to $\Lambda\actson (Y,\nu)$ and $\Gamma\actson (Y_{0},\nu_{0}).$ In particular, $\Gamma\actson (X,\mu)$ is strongly ergodic relative to $\Gamma\actson (Y_{0},\nu_{0})$ and $\Gamma\actson (X,\mu)$ is strongly ergodic relative to $\Gamma\actson (Y,\nu).$

\end{cor}

We remark that there is another approach to entropy for actions of nonamenable groups called Rokhlin entropy, first investigated by Seward in \cite{SewardKrieger}. Rokhlin entropy is easy to define and is defined for actions of arbitrary groups, but it is extremely hard to compute. There are no known instances where one can show that an action has positive Rokhlin entropy without knowing it has positive sofic entropy. In every case where the Rokhlin entropy has been computed and it is positive the computation has been done by first computing the sofic entropy, then using the general fact that sofic entropy is a lower bound for Rokhlin entropy, and finally showing (by methods that varies from case to case) that the sofic entropy is an upper bound for the Rokhlin entropy. Thus, in our opinion, there has yet to be a satisfactory, explicit computation of Rokhlin entropy which does not go through computing sofic entropy. In an analogous manner one can define the Rokhlin Pinsker factor and the outer Rokhlin Pinsker factor for a probability measure-preserving action of an arbitrary group. Alpeev in \cite{AALP} showed that any probability measure-preserving action is weakly mixing over its Rokhlin Pinsker factor. It is appears to be unknown if any probability measure-preserving action is in fact mixing over its Rokhlin  Pinsker factor. It is also unknown if the conclusion of Corollary \ref{C:IntroRepstuff} holds over the Rokhlin Pinsker factor, or if any action of a nonamenable group is strongly ergodic over its Rokhlin Pinsker factor. It appears to be very difficult to deduce any spectral properties of actions from positivity of Rokhlin entropy.

\section{Proof of The Main Theorem}

We start with the definition of a sofic group. For $n\in\NN,$ we use $u_{n}$ for the uniform measure on $\{1,\dots,n\}.$

\begin{definition}\emph{Let $\Gamma$ be a countable, discrete group. A} sofic approximation \emph{of $\Gamma$ is a sequence of functions $\sigma_{i}\colon\Gamma\to S_{d_{i}}$ (not assumed to be homomorphisms) so that}
\begin{align*}
&u_{d_{i}}(\{1\leq j\leq d_{i}:\sigma_{i}(gh)(j)=\sigma_{i}(g)\sigma_{i}(h)(j)\})\to 1\mbox{ \emph{for all $g,h\in\Gamma$},}\\
&u_{d_{i}}(\{1\leq j\leq d_{i}:\sigma_{i}(g)(j)\ne j\})\to 1\mbox{ \emph{for all $g\in\Gamma.$}}
\end{align*}
\emph{We say that $\Gamma$ is} sofic \emph{if it has a sofic approximation.}

\end{definition}

Intuitively, the first condition of a sofic approximation says that we have an ``almost actions'' of $\Gamma$ on the finite set $\{1,\dots,d_{i}\}$ and the second condition of a sofic approximation says that this action is ``almost free.'' Since finite groups can be characterized as those groups which act freely on finite sets we may view soficity as the analogue of finiteness one obtains by replacing the exact algebra with approximate algebra.  We now turn to some preliminaries needed for the definition of entropy and extension entropy. It will be important in this paper that we can reduce the computation of entropy (and extension entropy) to generating observables.

\begin{definition}\emph{Let $(X,\mathcal{M},\mu)$ be a standard probability space. Let $\mathcal{S}$ be a subalgebra of $\mathcal{M}$ (here $\mathcal{S}$ is not necessarily a $\sigma$-algebra).} A finite $\mathcal{S}$-measurable observable \emph{is a measurable map $\alpha\colon X\to A$ where $A$ is a finite set and $\alpha^{-1}(\{a\})\in \mathcal{S}$ for all $a\in A.$ If $\mathcal{S}=\mathcal{M},$ we simply call $\alpha$} a finite observable. \emph{Another finite $\mathcal{S}$-measurable observable $\beta\colon X\to B$ is said to} refine \emph{$\alpha,$ written $\alpha\leq \beta,$ if there is a  $\omega\colon B\to A$ so that $\omega(\beta(x))=\alpha(x)$ for almost every $x\in X.$ If $\Gamma$ is a countable discrete group and $\Gamma\actson(X,\mathcal{M},\mu)$ by measure-preserving transformations we say that $\mathcal{S}$ is generating if $\mathcal{M}$ is the $\sigma$-algebra generated by $\{gA:A\in\mathcal{S}\}$ (up to sets of measure zero).}
\end{definition}

Suppose we are given a standard probability space $(X,\mu),$ and a countable discrete group $\Gamma$ with $\Gamma\actson (X,\mu)$ by measure-preserving transformations. Given a finite observable $\alpha\colon X\to A,$ and a finite $F\subseteq\Gamma$ we let $\alpha^{F}\colon X\to A^{F}$ be defined by
\[\alpha^{F}(x)(g)=\alpha(g^{-1}x).\]

\begin{definition}\emph{Let $\Gamma$ be a countable discrete group and $\sigma\in S_{d}^{\Gamma}$ for some $d\in\NN.$ Let $(X,\mathcal{M},\mu)$ be a standard probability space and let $\mathcal{S}\subseteq\mathcal{M}$ be a subalgebra. Let $\alpha\colon X\to A$ be a finite $\mathcal{S}$ measurable-observable. Given $F\subseteq\Gamma$ finite, and $\delta>0,$ we let $\AP(\alpha,F,\delta,\sigma)$ be all $\phi\colon\{1,\dots,d\}\to A^{F}$ so that}
\[\sum_{a\in A^{F}}\left|u_{d_{i}}(\phi^{-1}(\{a\}))-\mu((\alpha^{F})^{-1}(\{a\}))\right|<\delta.\]
\[u_{d_{i}}(\{1\leq j\leq d_{i}:\phi(j)(g)=\phi(\sigma_{i}(g)^{-1}(j))(e)\})<\delta\mbox{\emph{ for all $g\in F.$}}\]

\end{definition}

We can now define extension entropy. The following definition was given by Kerr in \cite{KerrPartition} and is a natural generalization of Bowen's original definition of measure entropy in \cite{Bow}. For notation, if $f\colon B\to A$ and $C\subseteq B^{X}$ for some set $X,$ we let $f\circ C=\{f\circ \phi:\phi\in C\}.$

\begin{definition}\emph{Let $\Gamma$ be a countable discrete sofic group with sofic approximation $\Sigma=(\sigma_{i}\colon\Gamma\to S_{d_{i}}).$ Let $(X,\mathcal{M},\mu)$ be a standard probability space and $\Gamma\actson (X,\mathcal{M},\mu)$ by measure-preserving transformations. Let $\mathcal{S},\mathcal{T}$ be  subalgebras of $\mathcal{M}.$ Assume that $\mathcal{S}\subseteq\mathcal{T}.$  Let $\alpha\colon X\to A$ be a finite $\mathcal{S}$-measurable observable and let $\beta\colon X\to B$ be a $\mathcal{T}$-measurable observable refining $\alpha$ and $\omega\colon B\to A$ as in the definition of $\alpha\leq\beta.$ For a finite $F\subseteq \Gamma,$ we define}
\[\widetidle{\omega}\colon B^{F}\to A\]
 \emph{by}
 \[\widetilde{\omega}(b)=\omega(b(e)).\]
 \emph{We set}
\[h_{\Sigma,\mu}(\alpha:\beta,F,\delta)=\limsup_{i\to\infty}\frac{1}{d_{i}}\log\left|\widetilde{\omega}\circ (\AP(\alpha,F,\delta,\sigma_{i}))\right|\]
\[h_{\Sigma,\mu}(\alpha:\beta,\Gamma)=\inf_{\substack{ F\subseteq\Gamma \textnormal{ finite},\\ \delta>0}}h_{\Sigma,\mu}(\alpha:\beta,F,\delta).\]
\emph{We then set}
\[h_{\Sigma,\mu}(\alpha:\mathcal{T},\Gamma)=\inf_{\alpha\leq \beta}h_{\Sigma,\mu}(\alpha;\beta,\Gamma)\]
\[h_{\Sigma,\mu}(\mathcal{S}:\mathcal{T},\Gamma)=\sup_{\alpha}h_{\Sigma,\mu}(\alpha:\mathcal{S})\]
\emph{where the last infimum  is over all $\mathcal{T}$-measurable observables, and the supremum is over all $\mathcal{S}$-measurable observables. }
\end{definition}
It is known that
\[h_{\Sigma,\mu}(\mathcal{S}:\mathcal{T},\Gamma)\]
only depends upon the $\Gamma$-invariant sigma-algebra of sets generated by $\mathcal{S},\mathcal{T}.$ It is known that if $\mathcal{S}$ is a complete $\Gamma$-invariant subsigma-algebra of $\mathcal{M},$ then there is factor map $\pi\colon (X,\mu)\to (Y,\nu)$ so that
\[\mathcal{S}=\{\pi^{-1}(A):A\subseteq Y \mbox{ is $\nu$-measurable}\}.\]
Conversely, if we are given a factor map $\pi\colon (X,\mu)\to (Y,\nu)$ then
\[\{\pi^{-1}(A):A\subseteq Y\mbox{ is $\nu$-measurable}\}\]
is a $\Gamma$-invariant subsgima-algebra. Because of this, we will frequently blur the lines between observables, $\Gamma$-invariant subsigma-algebras and factors. Thus if $\mathcal{A}$ is  $\Gamma$-invariant sigma-algebra of measurable set in $X,$ and $Y$ is the factor generated by this algebra, we shall use
\[h_{(\sigma_{i})_{i},\mu}(Y:X,\Gamma)\]
for
\[h_{(\sigma_{i})_{i},\mu}(\mathcal{A}:\mathcal{M},\Gamma)\]
where $\mathcal{M}$ is the measurable subsets of $X.$ We will call $h_{(\sigma_{i})_{i},\mu}(Y:X,\Gamma)$ the entropy of $Y$ in the presence of $X$ (with respect to $(\sigma_{i})_{i}$) or the extension entropy of the extension
\[\Gamma\actson (X,\mu)\to \Gamma\actson (Y,\nu).\]
By \cite{KerrPartition} we have
\[h_{(\sigma_{i})_{i},\mu}(X:X,\Gamma)=h_{(\sigma_{i}),\mu}(X,\Gamma).\]

	As in \cite{Me6}, we need to use a way to compute the entropy of $Y$ in the presence of $X$ using topological models for
\[X\to Y.\]
For this, we recall some terminology from \cite{Me6}. Let $X$ be a Polish space and $\Gamma\actson X$ an action of a countable discrete group $\Gamma$ by homeomorphisms. We say that a continuous pseudometric $\Delta$ is dynamically generating if for every open subset $U$ of $X$ and every $x\in U,$ there is a $\delta>0$ and a finite $F\subseteq \Gamma$ so that
\[\bigcap_{g\in F}\{y\in X:\Delta(gx,gy)<\delta\}\subseteq U.\]
We note here that our definition of dynamically generating contains the assumption that $\Delta$ is continuous. Let $(A,\Delta)$ be a pseudometric space. For subsets $C,B$ of $A,$ and $\varepsilon>0$ we say that $C$ is $\varepsilon$-contained in $B$ and write $C\subseteq_{\varepsilon}B$ if for all $c\in C,$ there is a $b\in B$ so that $\Delta(c,b)<\varepsilon.$ We say that $S\subseteq A$ is $\varepsilon$-dense if $A\subseteq_{\varepsilon}S.$ We use $S_{\varepsilon}(A,\Delta)$ for the smallest cardinality of a $\varepsilon$-dense subset of $A.$ If $C\subseteq_{\delta}B$ are subsets of $A,$ then
\[S_{2(\varepsilon+\delta)}(C,\Delta)\leq S_{\varepsilon}(B,\Delta).\]
We say  that $N\subseteq A$ is $\varepsilon$-separated if for every $n_{1}\ne n_{2}$ in $N$ we have $\Delta(n_{1},n_{2})>\varepsilon.$ We use $N_{\varepsilon}(A,\Delta)$ for the smallest cardinality of a $\varepsilon$-separated subset of $A.$ Note that
\begin{equation}\label{E:separationspanning}
N_{2\varepsilon}(A,\Delta)\leq S_{\varepsilon}(A,\Delta)\leq N_{\varepsilon}(A,\Delta),
\end{equation}
and that if $A\subseteq B,$ then
\[N_{\varepsilon}(A,\Delta)\leq N_{\varepsilon}(B,\Delta).\]

\begin{definition}\emph{ Let $\Gamma$ be a countable discrete group and $X$ a Polish space with $\Gamma\actson X$ by homeomorphisms. Let $\Delta$ be a bounded pseudometric on $X.$ For a function $\sigma\colon \Gamma\to S_{d},$ for some $d\in \NN,$ a finite $F\subseteq \Gamma,$ and a $\delta>0$ we let $\Map(\Delta,F,\delta,\sigma)$ be all functions $\phi\colon \{1,\dots,d\}\to X$ so that}
\[\max_{g\in F}\Delta_{2}(\phi\circ \sigma(g),g\phi)<\delta.\]
\end{definition}
Given a  Polish space $X$, a finite $L\subseteq C_{b}(X),$ a $\delta>0,$ and $\mu\in \Prob(X)$ let
\[U_{L,\delta}(\mu)=\bigcap_{f\in L}\left\{\nu\in\Prob(X):\left|\int f\,d\nu-\int f\,d\mu\right|<\delta\right\}.\]
Then $U_{L,\delta}(\mu)$ form a basis of neighborhoods of $\mu$ for the weak topology. Here $C_{b}(X)$ is  the space of bounded continuous functions on $X.$
\begin{definition}\emph{Suppose that $\mu$ is a $\Gamma$-invariant Borel probability measure on $X.$ For $F\subseteq\Gamma$ finite, $\delta>0$ and $L\subseteq C_{b}(X)$ finite, and $\sigma\colon\Gamma\to S_{d}$ for some $d\in\NN$  we let $\Map_{\mu}(\Delta,F,\delta,L,\sigma)$ be the set of all $\phi\in\Map(\Delta,F,\delta,\sigma)$ so that}
\[\phi_{*}(u_{d})\in U_{L,\delta}(\mu).\]
\emph{for all $f\in L.$}\end{definition}

Recall that if $X,Y$ are Polish spaces a continuous, surjective map $\pi\colon X\to Y$ is a \mbox{quotient map} if  $\{E\subseteq X:\pi^{-1}(E)\mbox{ is open}\}$ equals the set of open subsets of $X.$

\begin{definition}\emph{Let $\Gamma$ be a countable discrete sofic group with sofic approximation $\sigma_{i}\colon\Gamma\to S_{d_{i}}.$ Let $X$ and $Y$ be Polish spaces with $\Gamma\actson X,\Gamma\actson Y$ by homeomorphisms. Suppose that there  exists a $\Gamma$-equivariant quotient map $\pi\colon X\to Y$. Let $\mu,\nu$ be $\Gamma$-invariant Borel probability measures on $X,Y$ with $\pi_{*}\mu=\nu.$ Let $\Delta_{X},\Delta_{Y}$ be bounded, dynamically generating pseudometrics for $X,Y.$ Inductively define}
\[h_{(\sigma_{i})_{i},\mu}(\Delta_{Y}:\Delta_{X},\varepsilon,F,\delta,L)=\limsup_{i\to\infty}\frac{1}{d_{i}}\log N_{\varepsilon}(\pi^{d_{i}}(\Map_{\mu}(\Delta_{X},F,\delta,L,\sigma_{i})),\Delta_{Y})\]
\[h_{(\sigma_{i})_{i},\mu}(\Delta_{Y}:\Delta_{X},\varepsilon)=\inf_{\substack{\textnormal{ finite} F\subseteq\Gamma,\\ \delta>0,\\ \textnormal{ finite} L\subseteq C_{b}(X)}}h_{(\sigma_{i})_{i},\mu}(\Delta_{Y}:\Delta_{X},\varepsilon,F,\delta,L)\]
\[h_{(\sigma_{i})_{i},\mu}(\Delta_{Y}:\Delta_{X})=\sup_{\varepsilon>0}h_{(\sigma_{i})_{i},\mu}(\Delta_{Y}:\Delta_{X},\varepsilon).\]

\end{definition}

We wish to prove that
\[h_{(\sigma_{i})_{i},\mu}(\Delta_{Y}:\Delta_{X})=h_{(\sigma)_{i},\mu}(Y:X,\Gamma),\]
for any choice of dynamically generating pseudometric $\Delta_{Y},\Delta_{X}.$ Much of the proof follows that of Theorem 3.12 in \cite{Me6}.  The following Lemma follows exactly as in Lemma 3.9 of \cite{Me6}. For notation, if $X$ is a Polish space and $\Gamma\actson X$ we let $\Prob_{\Gamma}(X)$ be the set of $\Gamma$-invariant, Borel probability measures on $X.$

\begin{lemma}\label{L:siwchingspsm} Let $\Gamma$ be a countable, discrete, sofic group with sofic approximation $\sigma_{i}\colon\Gamma\to S_{d_{i}}$. Let $X,Y$ be Polish spaces with $\Gamma\actson X,\Gamma\actson Y$ by homeomorphisms. Suppose there exists a topological factor map $\pi\colon X\to Y$ and $\mu\in \Prob_{\Gamma}(X),\nu\in \Prob_{\Gamma}(Y)$ with $\pi_{*}\mu=\nu.$ For any pair of dynamically generating pseudometrics $\Delta_{Y},\Delta_{X}$ on $Y,X,$ we can find compatible metrics $\Delta_{Y}',\Delta_{X}'$ so that
\[h_{(\sigma_{i})_{i},\mu}(\Delta_{Y}':\Delta_{X}')=h_{(\sigma_{i})_{i},\mu}(\Delta_{Y}:\Delta_{X}).\]

\end{lemma}

We need the following Lemma, which gives us a canonical way of producing microstates for a factor.

\begin{lemma}\label{L:asdlghalsdhn} Let $\Gamma$ be  a countable, discrete, sofic group with sofic approximation $\sigma_{i}\colon\Gamma\to S_{d_{i}}.$ Let $X,Y$ be Polish spaces with $\Gamma\actson X,\Gamma\actson Y$ by homeomorphisms. Suppose there exists a topological factor map $\pi\colon X\to Y$ and $\mu\in \Prob_{\Gamma}(X),\Prob_{\Gamma}(Y)$ with $\pi_{*}\mu=\nu.$ Fix dynamically generating pseudometrics $\Delta_{Y},\Delta_{X}$ on $Y,X.$ Then, for any finite $F\subseteq\Gamma,L\subseteq C_{b}(Y)$ and $\delta>0,$ there exists finite $F'\subseteq \Gamma,L'\subseteq C_{b}(X)$ and $\delta'>0$ so that for all sufficiently large $i,$
\[\pi\circ \Map_{\mu}(\Delta_{X},F,',\delta',L',\sigma_{i})\subseteq \Map_{\nu}(\Delta_{Y},F,\delta,L,\sigma_{i}).\]

\end{lemma}

\begin{proof}  Fix finite $F\subseteq \Gamma,L\subseteq C_{b}(X)$ and $\delta>0.$ Let $M$ be the diameter of $\Delta_{Y}.$ By Prokhorov's theorem, we may find a compact $K\subseteq X$ so that $\mu(X\setminus K)<\delta.$ Choose $\eta>0$  and a finite $E\subseteq \Gamma$ so that if $x,y\in K$ and
\[\max_{h\in E}\Delta_{X}(hx,hy)<\eta,\]
then $\Delta_{Y}(\pi(x),\pi(y))<\delta.$ By Lemma 3.10 in \cite{Me6}, we may find a $F_{0}'\subseteq\Gamma, L_{0}'\subseteq \Gamma$ finite and $\delta_{0}'>0$ so that
\[\phi_{*}(u_{d_{i}})(X\setminus K)\leq 2\delta,\mbox{ for all $\phi\in\Map_{\mu}(\Delta_{X},F_{0}',\delta_{0}',L_{0}',\sigma_{i}).$ }\]
Let $\delta'>0$ depend upon $\delta,\eta$ in manner to be determined later, we will at least assume that $\delta'<\min(\delta,\delta_{0}').$ Set
\[L'=L_{0}'\cup\{f\circ \pi:f\in L\},\]
\[F'=F_{0}'\cup EF.\]
Now suppose that $\phi\in\Map_{\mu}(\Delta_{X},F',\delta',L',\sigma_{i}),$ as $\phi_{*}(u_{d_{i}})\in U_{L',\delta'}(\mu),$ we have $(\pi\circ \phi)_{*}(u_{d_{i}})\in U_{L,\delta}(\nu).$ Fix $g\in F$ and let
\[C=\bigcap_{h\in E}\{1\leq j\leq d_{i}:\Delta_{X}(hg\phi(j),h\phi(\sigma_{i}(g)(j))<\eta\},\]
\[D=\phi^{-1}(X\setminus K).\]
For all $h\in E,$ we have
\begin{align*}
\Delta_{X,2}(hg\phi,h\phi\circ \sigma_{i}(g))&\leq \Delta_{X,2}(hg\phi,\phi\circ \sigma_{i}(hg))+\Delta_{X,2}(\phi\circ\sigma_{i}(hg),\phi\circ\sigma_{i}(h)\sigma_{i}(g))\\
&+\Delta_{X,2}(h\phi\circ\sigma_{i}(g),\phi\circ\sigma_{i}(h)\sigma_{i}(g))\\
&=\Delta_{X,2}(hg\phi,\phi\circ \sigma_{i}(hg))+\Delta_{X,2}(\phi\circ\sigma_{i}(hg),\phi\circ\sigma_{i}(h)\sigma_{i}(g))+\Delta_{X,2}(h\phi,\phi\circ\sigma_{i}(h))\\
&\leq 2\delta'+Mu_{d_{i}}(\{j:\sigma_{i}(hg)(j)\ne \sigma_{i}(h)\sigma_{i}(g)(j)\}).
\end{align*}
By soficity we have
\[\Delta_{X,2}(hg\phi,h\phi\circ \sigma_{i}(g))\leq 3\delta',\]
if $i$ is sufficiently large. Thus for all sufficiently large $i,$ we have
\[u_{d_{i}}(D^{c}\cup C^{c})\leq 2\delta+9\left(\frac{\delta'}{\eta}\right)^{2}|E|.\]
We may choose $\delta'$ sufficiently small so that
\[u_{d_{i}}(D^{c}\cup C^{c})\leq 3\delta.\]
We then have for all sufficiently large $i,$
\[\Delta_{Y,2}(g\pi\circ \phi,\pi\circ \phi\circ \sigma_{i}(g))^{2}\leq 3\delta M+\delta^{2}.\]
So  $\pi\circ \phi\in \Map_{\nu}(\Delta_{Y},F,L,(3\delta M+\delta^{2})^{1/2},\sigma_{i}),$ for all sufficiently large $i.$

\end{proof}

Before we prove that extension entropy can be expressed via dynamically generating pseudometrics, we need some more notation. If we are given $\sigma\in S_{d}^{\Gamma},$ and $\phi\in A^{d}$ we define, for a finite $F\subseteq \Gamma,$
\[\phi^{F}_{\sigma}\colon\{1,\cdots,d\}\to A^{F}\]
by
\[(\phi^{F}_{\sigma})(j)(g)=\phi(\sigma(g)^{-1}j).\]
Given a Polish space $X$ and $\mu\in \Prob_{\Gamma}(X),$ we let $\CO_{\mu}$ the set of all Borel subsets of $X$ so that
\[\mu(\Int E)=\mu(\overline{E}).\]

\begin{theorem}\label{T':topmodyo} Let $\Gamma$ be a countable discrete sofic group with sofic approximation $\sigma_{i}\colon\Gamma\to S_{d_{i}}.$ Let  $X$ and $Y$ be Polish spaces with $\Gamma\actson X,\Gamma\actson Y$ by homeomorphisms. Suppose that there exists a $\Gamma$-equivariant quotient map $\pi\colon X\to Y$. Let $\mu,\nu$ be $\Gamma$-invariant Borel probability measures on $X,Y$ with $\pi_{*}\mu=\nu.$ Let $\Delta_{X},\Delta_{Y}$ be  bounded, dynamically generating pseudometrics for $X,Y.$ Then
\[h_{(\sigma_{i})_{i},\mu}(\Delta_{Y}:\Delta_{X},\Gamma)=h_{(\sigma_{i})_{i},\mu}(Y:X,\Gamma).\]

\end{theorem}

\begin{proof}

	By  Lemma \ref{L:siwchingspsm}, we may assume that $\Delta_{Y},\Delta_{X}$ are compatible. Let $M_{X},M_{Y}$ be the diameters of $\Delta_{X},\Delta_{Y}.$ We first prove that
\[h_{(\sigma_{i})_{i},\mu}(\Delta_{Y}:\Delta_{X},\Gamma)\leq h_{(\sigma_{i})_{i},\mu}(Y:X,\Gamma).\]
Let $\varepsilon>0,$ since $Y$ is Polish we can apply Prokhorov's Theorem to find a compact $K\subseteq Y$ so that
\[\nu(Y\setminus K)<\varepsilon.\]
By compactness of $K,$ we may find $y_{1},\dots,y_{n}\in K$ and $\varepsilon>\delta_{1},\dots,\delta_{n}>0$ so that
\[K\subseteq \bigcup_{j=1}^{n}B_{\Delta_{Y}}(y_{j},\delta_{j}),\]
\[B_{\Delta_{Y}}(y_{j},\delta_{j})\in \mathcal{CO}_{\nu}.\]
Set
\[E=Y\setminus \bigcup_{j=1}^{n}B_{\Delta_{Y}}(y_{j},\delta_{j}),\]
and define
\[\alpha\colon Y\to \{0,1\}^{n+1}\]
by
\[\alpha(y)(k)=\begin{cases}
\chi_{B_{\Delta_{Y}}(y_{j},\delta_{j})}(y),& \textnormal{ if $1\leq k\leq n $}\\
\chi_{E}(y), &\textnormal{ if $k=n+1$}
\end{cases}.\]

Let $\beta\colon  X\to B$ be any $\mathcal{CO}_{\mu}(X)$-measurable observable which refines $\alpha\circ \pi.$ Since $\beta$ refines $\alpha\circ \pi,$ we can find a $\omega\colon B\to \{0,1\}^{n+1}$ such that $\omega\circ \beta=\alpha\circ \pi.$ Suppose we are given a finite $F\subseteq \Gamma$ and a $\delta>0.$ By Lemma 3.11 in \cite{Me6} we may find finite $F'\subseteq\Gamma, L'\subseteq C_{b}(X)$ and a $\delta'>0$ so that
\[\beta^{F}_{\sigma_{i}}(\Map_{\mu}(\Delta_{X},F',\delta',L',\sigma_{i}) )\subseteq \AP(\beta,F,\delta,\sigma_{i}).\]
By Lemma 3.10 in \cite{Me6} and Lemma \ref{L:asdlghalsdhn} we may assume that $L'$ is sufficiently large so that
\[(\pi\circ \phi)_{*}(u_{d_{i}})(Y\setminus E)\leq 2\varepsilon,\]
for all $\phi\in \Map_{\mu}(\Delta_{X},F',\delta',L',\sigma_{i})$ and all sufficiently large $i.$ Choose an index set $S$ and elements $\{\phi_{s}\}_{s\in S}$ so that
\[\phi_{s}\in \Map_{\mu}(\Delta_{X},F',L',\delta',\sigma_{i})\mbox{ for all $s\in S$},\]
\[\{\alpha\circ \pi\circ \phi_{s}:s\in S\}=\alpha\circ \pi\circ \Map_{\mu}(\Delta_{X},F',L',\delta',\sigma_{i}),\]
\[\alpha\circ \pi\circ \phi_{s}\ne \alpha\circ \pi\circ \phi_{s'},\mbox{ for $s\ne s'$ in S}.\]
As $\omega\circ \beta\circ \phi_{s}=\alpha\circ \pi\circ \phi_{s},$ we have
\[|S|\leq |\widetilde{\omega}\circ (\AP(\beta,F',\delta',\sigma_{i}))|.\]
Now let $\phi\in \Map_{\mu}(\Delta_{X},F',\delta',\sigma_{i}).$ Choose an $s\in S$ so that
\[\alpha\circ \pi\circ \phi_{s}=\alpha\circ \pi\circ \phi.\]
Then
\[\Delta_{Y,2}(\pi\circ \phi,\pi\circ \phi_{s})^{2}\leq 4M_{Y}^{2}\varepsilon+\frac{1}{d_{i}}\sum_{j:\pi(\phi(j)),\pi(\phi_{s}(j))\notin E}\Delta_{Y}(\pi(\phi(j)),\pi(\phi_{s}(j)))^{2}.\]
If $\phi(j),\phi_{s}(j)$ are not in $E,$ then the fact that $\alpha^{F}(\pi(\phi(j)))(e)=\alpha^{F}(\pi(\phi_{s}(j)))(e)$ implies that
\[\Delta_{Y}(\pi(\phi(j)),\pi(\phi_{s}(j)))<\varepsilon.\]
So
\[\Delta_{Y,2}(\pi\circ \phi, \pi\circ\phi_{s})^{2}<4M_{Y}^{2}\varepsilon+\varepsilon^{2}.\]
Thus
\[h_{(\sigma_{i})_{i},\mu}(\Delta_{Y}:\Delta_{X};2(4M_{Y}^{2}\varepsilon+\varepsilon^{2})^{1/2})\leq h_{(\sigma_{i})_{i},\mu}(\alpha:\beta;F,\delta).\]
Taking the infimum over $\beta,F,\delta$ we find
\[h_{(\sigma_{i})_{i},\mu}(\Delta_{Y}:\Delta_{X};2(4M_{Y}^{2}\varepsilon+\varepsilon^{2})^{1/2})\leq h_{(\sigma_{i})_{i},\mu}(\alpha:\CO_{\mu})\leq h_{(\sigma_{i}),\mu}(Y:X,\Gamma).\]
And letting $\varepsilon\to 0$ shows that
\[h_{(\sigma_{i})_{i},\mu}(\Delta_{Y}:\Delta_{X},\Gamma)\leq  h_{(\sigma_{i}),\mu}(Y:X,\Gamma).\]

	We now turn to the reverse inequality. Let $\alpha\colon Y\to A$ be a $\CO_{\nu}$-measurable observable. Fix $\kappa>0$ and let $\kappa'>0$ depend upon $\kappa$ in a manner to be determined later. By Lemma 3.10 in \cite{Me6} we may choose an $\eta>0$ and $L_{0}\subseteq C_{b}(Y)$ finite so that if $\zeta\in \Prob(Y)$ and
	\[\left|\int_{Y} f\,d\mu-\int_{Y} f\,d\zeta\right|<\eta,\]
for all $f\in L_{0},$ then
\[|\zeta(\alpha^{-1}(\{a\}))-\mu(\alpha^{-1}(\{a\}))|<\kappa',\mbox{ for all $a\in A$,}\]
\[\zeta(\mathcal{O}_{\eta}(\alpha^{-1}(\{a\}))\setminus \alpha^{-1}(\{a\})_{\eta})<\kappa'\mbox{ for all $a\in A.$}\]
Let $F'\subseteq \Gamma,L'\subseteq C_{b}(X)$ be given finite sets and $\delta'>0$ be given. We assume that
\[L'\supseteq\{f\circ \pi:f\in L_{0}\}.\]
By Lemma 3.11 in \cite{Me6}, we may choose a refinement $\beta\colon X\to B$ of $\alpha,$ a finite $F\subseteq\Gamma$ and a $\delta>0$ so that if $\beta^{F}\circ \phi\in \AP(\beta,F,\delta,\sigma_{i}),$ then $\phi\in \Map_{\mu}(\Delta_{X},F',L',\delta',\sigma_{i}).$ Choose $\omega\colon B\to A$ so that $\beta\circ \omega=\alpha\circ \pi$ and choose a map $s\colon B^{F}\to X$ so that $\id=\beta^{F}\circ s.$ By construction, if $\phi\in \AP(\beta,F',\delta',\sigma_{i}),$ then we have
\[s\circ \phi\in\Map_{\mu}(\Delta_{X},F,\delta,L,\sigma_{i}).\]
Let $\varepsilon>0$ be sufficiently small depending upon $\eta$ in a manner to be determined later. Let $T$ be an index set and $\{\phi_{t}\}_{t\in T}$ be such that
\[\phi_{t}\in\AP(\beta,F',\delta',\sigma_{i})\mbox{ for all $t\in T$},\]
\[\{\pi\circ s\circ\phi_{t}:t\in T\}\mbox{ is $\varepsilon$-dense in $\{\pi\circ s\circ \phi:\phi\in \AP(\beta,F',\delta',\sigma_{i})\}$}\]
\[\pi\circ s\circ \phi_{t}\ne \pi\circ s\circ \phi_{t'}\mbox{ if $t\ne t'.$}\]
We may choose such a $T$ with
\[|T|\leq S_{\varepsilon/2}(\pi\circ \Map_{\mu}(\Delta_{X},F',L',\delta',\sigma_{i})).\]
Then
\[\widetilde{\omega}\circ (\AP(\beta,F,\delta,\sigma_{i}))\subseteq\bigcup_{t\in T}\alpha\circ \pi\circ (B_{\Delta_{Y,2}}(\pi\circ s\circ\phi_{t},\varepsilon)\cap \Map_{\mu}(\Delta_{X},F',L',\delta',\sigma_{i})).\]
We thus have to bound $|\alpha\circ \pi\circ (B_{\Delta_{Y,2}}(\pi\circ s\circ\phi_{t},\varepsilon)\cap \Map_{\mu}(\Delta_{X},F',L',\delta',\sigma_{i}))|$ from above. Fix $t\in T,$ suppose that $\phi\in\Map_{\mu}(\Delta_{X},F',\delta',L',\sigma_{i})$ and that $\Delta_{2,Y}(\pi\circ \phi,\pi\circ s\circ \phi_{t})<\varepsilon.$
Let
\[C=\bigcup_{a\in A}\{1\leq j\leq d_{i}:\pi(\phi(j))\in \mathcal{O}_{\eta}(\alpha^{-1})(\{a\}))\setminus \alpha^{-1}(\{a\})_{\eta}\}\cup\{1\leq j\leq d_{i}:\pi(s(\phi_{t}(j)))\in \mathcal{O}_{\eta}(\alpha^{-1}(\{a\})\setminus \alpha^{-1}(\{a\})_{\eta}\}.\]
If we choose $\kappa'$ sufficiently small, we then have that $u_{d_{i}}(C)\leq\kappa.$ Let
\[D=\{1\leq j\leq d_{i}:\Delta(\pi(\phi(j)),\pi(s(\phi_{t}(j))))\geq \sqrt{\varepsilon}\},\]
so $u_{d_{i}}(D)\leq \sqrt{\varepsilon}.$ For $j\in\{1,\dots d_{i}\}\setminus (C\cup D),$ we have that $\pi(s(\phi_{t}(j)))\in \mathcal{O}_{\sqrt{\varepsilon}}(\alpha^{-1}(\{a\})).$ Hence if we choose $\sqrt{\varepsilon}<\eta,$ then $\alpha(\pi(s(\phi_{t}(j))))=a$ for all $j\in\{1,\dots,d_{i}\}\setminus (C\cup D).$ So we can find a  $\mathcal{V}\subseteq\{1,\dots,d_{i}\}$ with $u_{d_{i}}(\mathcal{V})\geq 1-\kappa-\sqrt{\varepsilon}$ and $\alpha(\pi(s(\phi_{t}(j))))=\alpha(\pi(\phi(j)))$ for all $j\in \mathcal{V}.$ Thus
\begin{align*}
|\alpha\circ \pi\circ (B_{\Delta_{Y,2}}(\pi\circ s\circ\phi_{t},\varepsilon)\cap \Map_{\mu}(\Delta_{X},F',L',\delta',\sigma_{i}))|&\leq \sum_{\substack{\mathcal{V}\subseteq\{1,\dots,d_{i}\},\\ |\mathcal{V}|\leq (\kappa+\sqrt{\varepsilon})d_{i}}}|A|^{|\mathcal{V}|}\\
&\leq \sum_{l=1}^{\lfloor{(\kappa+\sqrt{\varepsilon})d_{i}\rfloor}}\binom{d_{i}}{l}|A|^{l}.
\end{align*}
If $\kappa+\sqrt{\varepsilon}<1/2$ then for all large $i$ we have
\[\binom{d_{i}}{l}\leq \binom{d_{i}}{\lfloor{\kappa+\sqrt{\varepsilon\rfloor}}d_{i}}.\]

So by Stirling's Formula the above sum is at most
\[R(\kappa+\sqrt{\varepsilon})d_{i}\exp(d_{i}H(\kappa+\sqrt{\varepsilon}))|A|^{\kappa d_{i}}\]
for some constant $R>0,$ where
\[H(t)=-t\log t-(1-t)\log(1-t)\mbox{ for $0\leq t\leq 1$}.\]
Thus
\[h_{(\sigma_{i})_{i},\mu}(\alpha;\mathcal{B})\leq h_{\Sigma,\mu}(\alpha;\beta,F,\delta)\leq H(\kappa+\sqrt{\varespilon})+\kappa\log|A|+h_{(\sigma_{i})_{i},\mu}(\Delta_{Y}:\Delta_{X};\varepsilon,F',\delta',L').\]
Taking the infimum over $F',\delta',L'$ and then letting $\kappa\to 0$ shows that
\[h_{(\sigma_{i})_{i},\mu}(\alpha;\mathcal{B})\leq h_{(\sigma_{i})_{i},\mu}(\Delta_{Y}:\Delta_{X},\varepsilon).\]
Letting $\varespilon\to 0$ and then taking the supremum over $\alpha$ shows that
\[h_{(\sigma_{i})_{i},\mu}(Y:X,\Gamma)=h_{(\sigma_{i})_{i},\mu}(\Delta_{Y}:\Delta_{X},\Gamma).\]

\end{proof}
If $Y$ is compact we use
\[h_{(\sigma_{i})_{i},\mu}(\Delta_{Y}:\Delta_{X},\infty)\]
for the quantity defined in the same manner as
\[h_{(\sigma_{i})_{i},\mu}(\Delta_{Y}:\Delta_{X})\]
replacing $\Delta_{Y,2}$ with $\Delta_{Y,\infty}.$ We apply similar remarks for
\[h_{(\sigma_{i})_{i},\mu}(\Delta_{Y},\varepsilon,F,\delta,L,,\infty)\]
\[h_{(\sigma_{i})_{i},\mu}(\Delta_{Y},\varepsilon,\infty).\]
\begin{proposition} Let $\Gamma$ be a countable discrete sofic group with sofic approximation $\sigma_{i}\colon\Gamma\to S_{d_{i}}.$ Let  $X$ and $Y$ be Polish spaces with $\Gamma\actson X,\Gamma\actson Y$ by homeomorphisms. Suppose that $Y$ is compact and  that there exists a $\Gamma$-equivariant quotient map $\pi\colon X\to Y$. Let $\mu,\nu$ be $\Gamma$-invariant Borel probability measures on $X,Y$ with $\pi_{*}\mu=\nu.$ Let $\Delta_{X},\Delta_{Y}$ be bounded dynamically generating pseudometrics for $X,Y.$ Then
\[h_{(\sigma_{i})_{i},\mu}(\Delta_{Y,\infty}:\Delta_{X},\infty,\Gamma)=h_{(\sigma_{i})_{i},\mu}(Y:X,\Gamma).\]

\end{proposition}

\begin{proof}
Fix $\varepsilon>0.$ Let $0<\varepsilon'<\varespilon$ be given. Let $E\subseteq Y$ be a finite $\varespilon$-dense set with respect to $\Delta_{Y},$ this may be done as $Y$ is compact. Fix $F\subseteq \Gamma,L\subseteq C_{b}(X)$ finite and $\delta>0.$ Let $S'$
be a $\varepsilon'$-dense subset of $\pi\circ \Map_{\mu}(\Delta_{X},F,\delta,L,\sigma_{i})$  with respect to $\Delta_{Y,2}$ of minimal cardinality. Choose a $T\subseteq \Map_{\mu}(\Delta_{X},F,\delta,L,\sigma_{i})$ so that $\pi\circ T=S$ and $|T|=|S|.$ Given $\phi\in \Map_{\mu}(\Delta_{X},F,L,\delta,\sigma_{i})$  choose a $\psi\in T$ with
\[\Delta_{2}(\pi\circ \phi,\pi\circ \psi)<\varepsilon'.\]
Let
\[C=\{1\leq j\leq d_{i}:\Delta_{Y}(\pi(\phi(j)),\pi(\psi(j)))<\varepsilon\}.\]
Then
\[u_{d_{i}}(C)\geq 1-\left(\frac{\varespilon'}{\varepsilon}\right)^{2}.\]
We thus see that
\[\pi\circ \Map_{\mu}(\Delta_{X},F,\delta,L,\sigma_{i})\subseteq_{\varepsilon,\Delta_{Y,\infty}}\bigcup_{\substack{\mathcal{V}\subseteq \{1,\dots,d_{i}\},\\ u_{d_{i}}(\mathcal{V})\geq 1-\left(\frac{\varespilon'}{\varepsilon}\right)^{2},\\ \psi \in S}} \{\phi\in Y^{d_{i}}:\phi\big|_{\mathcal{V}}=\psi, \phi\big|_{\mathcal{V}^{c}}\in E^{\mathcal{V}^{c}}\}.\]
As in the proof of the preceding theorem, we may find a $R>0$ so that
\[S_{2\varespilon}(\pi\circ \Map_{\mu}(\Delta_{X},F,\delta,L,\sigma_{i}),\Delta_{Y,\infty})\leq R\left(\frac{\varespilon'}{\varepsilon}\right)^{2}\exp\left(H\left(\left(\frac{\varespilon'}{\varepsilon}\right)^{2}\right)d_{i}\right)|E|^{\left(\frac{\varespilon'}{\varepsilon}\right)^{2}d_{i}}|S|.\]
Thus
\[h_{(\sigma_{i}),\mu}(\Delta_{Y}:\Delta_{X},\varepsilon,F,\delta,L,\infty)\leq H\left(\left(\frac{\varespilon'}{\varepsilon}\right)^{2}\right)+\left(\frac{\varespilon'}{\varepsilon}\right)^{2}\log|E|+h_{(\sigma_{i})_{i},\mu}(\Delta_{Y}:\Delta_{X},\varepsilon',F,\delta,L).\]
Taking the infimum over all $F,\delta,L$ we see that
\[h_{(\sigma_{i}),\mu}(\Delta_{Y}:\Delta_{X},\varepsilon,F)\leq  H\left(\left(\frac{\varespilon'}{\varepsilon}\right)^{2}\right)+\left(\frac{\varespilon'}{\varepsilon}\right)^{2}\log|E|+ h_{(\sigma_{i})_{i},\mu}(\Delta_{Y}:\Delta_{X},\varepsilon')\]
Letting $\varepsilon'\to 0$ and then $\varepsilon\to 0$ shows that
\[h_{(\sigma_{i})_{i},\mu}(\Delta_{Y,\infty}:\Delta_{X},\infty,\Gamma)\leq h_{(\sigma_{i})_{i},\mu}(Y:X,\Gamma).\]
Since the reverse inequality is trivial, the proof is complete.

\end{proof}

We use the definition of \emph{singularity} of representations of a $*$-algebra as in \cite{Me6} Definition 4.1. We first make a preliminary observation. Let $A$ be a $*$-algebra and $\rho_{j}\colon A\to B(\mathcal{H}_{j}),j=1,2$ be $*$-representations of $A$ on Hilbert spaces $\mathcal{H}_{j},j=1,2.$ Note that for $\xi\in\mathcal{H}_{1},$ we have that $\overline{A\xi}$ is singular with respect to $\mathcal{H}_{2}$ as a representation of $A$ if and only if $T(\xi)=0$ for all $T\in\Hom_{A}(\rho_{1},\rho_{2}).$

If $\mathcal{F}$ is a family of functions on $X$, we define the factor generated by $\mathcal{F}$ to be the factor associated to the sigma-algebra
 \[\{gf^{-1}(E):f\in \mathcal{F},g\in\Gamma,E\subseteq\CC\mbox{ is Borel} \}.\]
Given an extension
\[\Gamma\actson (X,\mu)\to \Gamma\actson (Y,\nu)\]
We define the factor generated by $\mathcal{F}$ over $Y$ to be the factor generated $\mathcal{F}\cup L^{\infty}(Y).$ If this factor is just $X$ itself, we say that $\mathcal{F}$ generates $X$ over $Y.$

Let $\Gamma$ be a countable, discrete, sofic group with sofic approximation $\sigma_{i}\colon\Gamma\to S_{d_{i}}.$ Let $X$ be a compact, metrizable space with $\Gamma\actson X$ by homeomorphisms. If $\phi\colon\{1,\dots,d_{i}\}\to X,$ and $f=\sum_{g\in\Gamma}f_{g}u_{g}\in C(X)\rtimes_{\textnormal{alg}}\Gamma,$ we define $\phi\rtimes\sigma_{i}(f)\in M_{d_{i}}(\CC)$ by
\[\phi\rtimes\sigma_{i}(f)=\sum_{g\in\Gamma}m_{f_{g}\circ \phi} \sigma_{i}(g).\]
We let $\tau_{\nu}$ be the trace on $L^{\infty}(Y,\nu)\rtimes_{\textnormal{alg}}\Gamma$ given by
\[\tau_{\nu}\left(\sum_{g\in\Gamma}f_{g}u_{g}\right)=\int_{Y}f_{e}\,d\nu.\]

\begin{theorem}\label{T:extensionnonsense} Let $\Gamma$ be a countable discrete group and $\sigma_{i}\colon\Gamma\to S_{d_{i}}$ a sofic approximation. Let $(X,\mu)$ be a standard probability space and $\Gamma\actson (X,\mu)$ a measure-preserving action.  Let $\mathcal{H}$ be a $L^{\infty}(Y)\rtimes_{\textnormal{alg}}\Gamma$-subrepresentation of $L^{2}(X)$ and let $(Z,\zeta)$ be the intermediate factor between $X$ and $Y$ generated by $\mathcal{H}\cup L^{\infty}(Y).$  If $\mathcal{H}$ is singular with respect to $L^{2}(Y,\ell^{2}(\Gamma))$ as a representation of $L^{\infty}(Y)\rtimes_{\textnormal{alg}}\Gamma,$ then
\[h_{(\sigma_{i})_{i},\mu}(Z:X,\Gamma)=h_{(\sigma_{i})_{i},\mu}(Y:X:\Gamma).\]
In particular,
\[h_{(\sigma_{i})_{i},\zeta}(Z,\Gamma)\leq h_{(\sigma_{i})_{i},\nu}(Y,\Gamma).\]
\end{theorem}

\begin{proof}

We let
\[\rho\colon L^{\infty}(Y)\rtimes_{\textnormal{alg}}\Gamma\to B(L^{2}(X))\]
be defined as before the theorem. It is clear that
\[h_{(\sigma_{i})_{i},\mu}(Z:X,\Gamma)\geq h_{(\sigma_{i})_{i},\mu}(Y:X,\Gamma)\]
so it suffices to show that
\[h_{(\sigma_{i})_{i},\mu}(Z:X,\Gamma)\leq h_{(\sigma_{i})_{i},\mu}(Y:X,\Gamma).\]

This is more or less implicit in \cite{Me6} Theorem 7.8, but we shall present a simplified proof.  Without loss of generality, suppose that $Y$ is a compact metrizable space, that $\Gamma\actson Y$ by homeomorphisms, and that $\nu$ is a Borel measure on $Y.$ Let $(\xi_{n})_{n=1}^{\infty}$ be a dense sequence in $\mathcal{H}.$

	We first reduce to the case that $\mathcal{H}$ is cyclic as a representation of $L^{\infty}(Y)\rtimes_{\textnormal{alg}}\Gamma.$ So suppose that we can prove the theorem in the case that $\mathcal{H}$ is cyclic. For $n\geq 1,$ let $(Z_{n},\zeta_{n})$ be the factor of  $X$ generated by $L^{\infty}(Y)$ and the functions
\[\{\xi_{1},\dots,\xi_{n}\}.\]
We use $(Z_{0},\zeta_{0})$ for $(Y,\nu).$ For $n\geq 1,$ let $\mathcal{K}_{n}$ be the smallest closed $L^{\infty}(Z_{n-1})\rtimes_{\textnormal{alg}}\Gamma$-invariant subspace of $L^{2}(X)$ containing $\xi_{n}.$  We claim that $\mathcal{K}_{n}$ is singular with respect to $L^{2}(Z_{n-1},\ell^{2}(\Gamma))$ as a representation of $L^{\infty}(Z_{n-1})\rtimes_{\textnormal{alg}}\Gamma.$ To see this suppose that
\[T\colon \mathcal{K}_{n}\to L^{2}(Z_{n-1},\ell^{2}(\Gamma))\]
is a nonzero, $L^{\infty}(Z_{n-1})\rtimes_{\textnormal{alg}}\Gamma$-equivariant, bounded, linear map. Then $T$ is $L^{\infty}(Y)\rtimes_{\textnormal{alg}}\Gamma$-equivariant. Since $L^{2}(Z_{n-1},\ell^{2}(\Gamma)$ embeds into
\[L^{2}(Y,\ell^{2}(\Gamma))^{\oplus\infty}\]
as a representation of $L^{\infty}(Y)\rtimes_{\textnormal{alg}}\Gamma$, and $\mathcal{H}$ is singular with respect to $L^{2}(Y,\ell^{2}(\Gamma)),$ our observation before the theorem shows that $T(\xi_{n})=0.$ Since $T$ is $L^{\infty}(Z_{n-1})\rtimes_{\textnormal{alg}}\Gamma$-equivariant and $\mathcal{K}_{n}$ is generated by $\xi_{n}$ as a representation of $L^{\infty}(Z_{n-1})\rtimes_{\textnormal{alg}}\Gamma$, we see that $T=0.$ Thus $\mathcal{K}_{n}$ is singular with respect to $L^{2}(Z_{n-1},\ell^{2}(\Gamma))$ as a representation of $L^{\infty}(Z_{n-1})\rtimes_{\textnormal{alg}}\Gamma.$

Since we are assuming, we can prove the Theorem in the cyclic case we see inductively that
\[h_{(\sigma_{i})_{i},\mu}(Z_{n}:X)= h_{(\sigma_{i})_{i},\mu}(Y:X)\]
for all $n\geq 1.$ As
\[h_{(\sigma_{i})_{i},\mu}(Z:X,\Gamma)\leq \liminf_{n\to\infty}h_{(\sigma_{i})_{i},\mu_{n}}(Z_{n}:X,\Gamma)\]
by Lemma 7.9 of \cite{Me6}, we have
\[h_{(\sigma_{i})_{i},\mu}(Z:X,\Gamma)\leq h_{(\sigma_{i})_{i},\mu}(Y:X,\Gamma).\]

 	Thus we may assume that $\mathcal{H}$ can be generated over $Y$ by a single $\xi\in L^{2}(X)\ominus L^{2}(Y).$ Without loss of generality, we may assume that
\[\|\xi\|_{2}\leq 1.\]
Arguing as in Theorem 7.8 of \cite{Me6} we may assume that $Z=\CC^{\Gamma}\times Y,$ that the factor map $\pi_{Y}\colon Z\to Y$ is projection onto the second factor,  that $\Gamma\actson Z$ is the product action where $\Gamma\actson \CC^{\Gamma}$ by Bernoulli shifts, and that $\xi$ is given by $\xi(z,y)=z(e).$ We may also assume that $X$ is a Polish space and that $Z$ is a continuous factor of $X$, let $\pi_{X}\colon X\to Z$ be the factor map.
 Let $\Delta_{Y}$ be a compatible metric on $Y,$ and let $\Delta_{Z}$ be the dynamical generated metric on $Z$ defined by
\[\Delta_{Z}((z_{1},y_{1}),(z_{2},y_{2}))=\min(|z_{1}(e)-z_{2}(e)|,1)+\Delta_{Y}(y_{1},y_{2}).\]
Fix a dynamically generating pseudometric $\Delta_{X}$ on $X.$ Using Proposition 4.2 in \cite{Me6} and the density of $C(Y)$ inside $L^{\infty}(Y,\nu)$ in the weak operator topology, it is not hard to argue that $L^{2}(Z)\ominus L^{2}(Y)$ is singular with respect to $L^{2}(Z,\zeta,\ell^{2}(\Gamma))$ as a representation of $C(Y)\rtimes_{\textnormal{alg}}\Gamma.$

	Let $1>\varepsilon>0,$ and let $0<\eta<\varepsilon$ be arbitrary and let $\eta>0.$ Since $\mathcal{H}$ is singular with respect to $L^{2}(Z,\zeta,\ell^{2}(\Gamma)),$  by Proposition 4.2 of \cite{Me6} we can find a $f\in C(Y)\rtimes_{\textnormal{alg}}\Gamma$ with
\[\|\rho(f)\|\leq 1\]
\[\tau_{\nu}(f^{*}f)<\eta^{2}\]
\[\|\rho(f^{*}f)\xi-\xi\|_{2}<\eta.\]
Write
\[f=\sum_{g\in E}f_{g}u_{g}.\]
with $E$ a finite subset of $\Gamma.$ Let $M>0.$ Choose a $G\in C_{c}(\CC)$ so that $G(z)=z$ if $|z|\leq M$ and $\|G\|_{C_{b}(\CC)}\leq M.$ We may suppose that $M$ is sufficiently large so that
\[\|G\circ \xi-\xi\|_{2}<\varepsilon\]
\[\mu(\{(z,y):|z(e)|\geq M\})\leq \varepsilon.\]
Since $\|\rho(f)\|\leq 1$ we have
\[\|\rho(f^{*}f)(G\circ \xi)-G\circ \xi\|_{2}<3\varepsilon.\]

	For $\phi\in X^{d_{i}},$  we set $\phi_{Y}=\pi_{Y}\circ\pi_{Z}\circ \phi,$ $\phi_{Z}=\pi_{Z}\circ \phi.$ Fix finite $F\subseteq\Gamma,L\subseteq C_{b}(X)$ and $\delta>0.$  By Lemma 7.7 of \cite{Me6}, we may assume that $F,\delta,L$ are chosen appropriately so that for all large $i$ and all  $\phi\in\Map_{\mu}(\Delta_{X},F,\delta,L,\sigma_{i})$ we have
	\[\|(\phi_{Y}\rtimes \sigma_{i})(f)^{*}(\phi_{Y}\rtimes \sigma_{i})(f)(G\circ \xi\circ \phi_{Z})-G\circ \xi\circ \phi_{Z}\|_{2}<4\varepsilon,\]
	\[\|(\phi_{Y}\rtimes \sigma_{i})(f)\|_{2}<2\eta,\]
	\[\|G\circ \xi\circ\phi_{Z}\|_{2}\leq 2,\]
	\[u_{d_{i}}(\{j:|\phi(j)|\geq M\})\leq 2\varepsilon.\]
For $\phi\in\Map_{\mu}(\Delta_{X},F,\delta,L,\sigma_{i}),$  set
\[p_{\phi}=\chi_{[0,\sqrt{\varpesilon}]}(|1-(\phi_{Y}\rtimes \sigma_{i})(f)^{*}(\phi_{Y}\rtimes \sigma_{i})(f)|)=\chi_{[1-\sqrt{\varpesilon},1+\sqrt{\varepsilon}]}(\phi_{Y}\rtimes \sigma_{i})(f)^{*}(\phi_{Y}\rtimes \sigma_{i})(f)).\]
Thus for all large $i,$
\[\tr(p_{\phi})\leq \frac{1}{(1-\sqrt{\varepsilon})}\|(\phi_{Y}\rtimes \sigma_{i})(f)\|_{2}^{2}\leq \frac{4\eta^{2}}{(1-\sqrt{\varepsilon})}.\]
By Lemma 2.7 of \cite{Me6} we find, for all large $i,$
\[S_{\phi}\subseteq Mp_{\phi}\Ball(\ell^{2}(d_{i},u_{d_{i}}))\]
 an $\varepsilon$-dense set with respect to $\|\cdot\|_{2}$ of cardinality  at most
\[M^{8\frac{\eta^{2}}{(1-\sqrt{\varepsilon})} d_{i}}\left(\frac{3+3\varepsilon}{\varpesilon}\right)^{\frac{8\eta^{2}}{(1-\sqrt{\varpesilon})} d_{i}}.\]

	Let $\kappa>0$  depend upon $\varepsilon,E$ in a manner to be determined later. Let $\varepsilon>\varepsilon'>0$ be such that if $a,b\in Y^{d_{i}}$ and $\Delta_{Y,\infty}(a,b)<\varepsilon'$ then
	\[\|f_{g}\circ a-f_{g}\circ b\|_{\infty}<\kappa\mbox{ for all $g\in E$}\]
 Choose a $D\subseteq \Map_{\mu}(\Delta_{X},F,\delta,L,\sigma_{i})$ so that $\{\psi_{Y}:\psi\in D\}$ is $\varespilon'$-dense in
\[\{\phi_{Y}:\phi\in \Map_{\mu}(\Delta_{X},F,\delta,L,\sigma_{i})\}\]
with respect to $\Delta_{Y,\infty}$ and
\[|D|=S_{\varepsilon'}(\{\phi_{Y}:\phi\in \Map_{\mu}(\Delta_{X},F,\delta,L,\sigma_{i})\},\Delta_{Y,\infty}).\]
 Suppose that $\phi,\psi\in  \Map_{\mu}(\Delta_{X},F,\delta,L,\sigma_{i})$ and that
\[\Delta_{Y,\infty}(\phi_{Y}, \psi_{Y})<\varepsilon'.\]
Then

\[\|(\phi_{Y}\rtimes \sigma_{i})(f)^{*}(\phi_{Y}\rtimes \sigma_{i})(f)(G\circ\xi\circ \phi_{Z})-(\psi_{Y}\rtimes \sigma_{i})(f)^{*}(\psi_{Y}\rtimes \sigma_{i})(f)(G\circ \xi\circ \phi_{Z})\|_{2}\leq\]
\[ \sum_{g,h\in E}\|(f_{g}\circ \phi_{Y})\cdot(f_{h}\circ \phi_{Y}\circ\sigma_{i}(g))\cdot(\sigma_{i}(g)(G\circ \xi\circ \phi_{Z}))-(f_{g}\circ \psi_{Y})\cdot(f_{h}\circ \psi_{Y} \circ \sigma_{i}(g))\cdot (\sigma_{i}(g)(G\circ \xi\circ \phi_{Z}))\|_{2}\leq\]
\[ 2|E|^{2}\kappa,\]
as
\[\|G\circ \xi\circ\phi_{Z}\|_{2}\leq 2.\]
Now choose $\kappa$ so that
\[2|E|^{2}\kappa<\varepsilon.\]

	Suppose that $\phi\in \Map_{\mu}(\Delta_{X},F,\delta,L,\sigma_{i})$ and that $\psi\in D$ has
\[\Delta_{Y,\infty}(\phi,\psi)<\varepsilon'.\]
By the above, we have
\[\|(\psi\rtimes\sigma_{i})(f)(G\circ \xi\circ \phi_{Z})-G\circ \xi\circ \phi_{Z}\|_{2}\leq \varepsilon+\|(\phi \rtimes \sigma_{i}(f))(G\circ \xi\circ \phi_{Z})-(G\circ \xi\circ \phi_{Z})\|_{2}<5\varepsilon.\]
If $i$ sufficiently large, we may choose a $w\in S_{\psi}$ so that
\[\|w-p_{\psi}(G\circ \xi\circ \phi_{Z})\|_{2}<\varepsilon.\]
Then
\begin{align*}
\|w-G\circ \xi\circ\pi_{Z}\circ \phi\|_{2}&\leq \varepsilon+\|(p_{\psi}-1)(G\circ \xi\circ \pi_{Z}\circ \phi)\|_{2}\\
&=\varespilon+\ip{(1-p_{\psi})(G\circ \xi\circ \phi_{Z}),G\circ \xi\circ \phi_{Z}}^{1/2}\\
&\leq \varepsilon+\frac{1}{\sqrt{\varepsilon}}\ip{|1-(\psi\rtimes\sigma_{i})(f)^{*}(\psi\rtimes\sigma_{i})(f)|^{2}(G\circ \xi\circ \pi_{Z}\circ \phi),G\circ \xi\circ\pi_{Z}\circ \phi}^{1/2}\\
&= \varepsilon+\frac{1}{\sqrt{\varepsilon}}\|(1-(\psi\rtimes\sigma_{i})(f)^{*}(\psi\rtimes\sigma_{i})(f))(G\circ \xi\circ \phi_{Z})\|_{2}\\
&\leq \varepsilon+\frac{4\varepsilon}{\sqrt{\varepsilon}}\\
& <5\sqrt{\varepsilon}.
\end{align*}
Given $v\in \CC^{d_{i}},$ define
\[\Theta_{v}\colon \{1,\dots,d_{i}\}\to \CC^{\Gamma}\]
by
\[\Theta_{v}(j)(g)=v(\sigma_{i}(g)^{-1}(j)).\]
Given $v\in \CC^{d_{i}},\beta\in Y^{d_{i}}$ define $\Gamma_{v,\beta}\in (\CC^{\Gamma}\times Y)^{d_{i}}$ by
\[\Gamma_{v,\beta}(j)=(\Theta_{v}(j),\beta(y)).\]
We now estimate
\[\Delta_{Z,2}(\phi,\Gamma_{w,\psi_{Y}}).\]
We have
\begin{align*}
\Delta_{Z,2}(\phi,\Gamma_{w,\psi_{Y}})&\leq \varespilon'+\left(\frac{1}{d_{i}}\sum_{j=1}^{d_{i}}\min(1,|w(\sigma_{i}(e)^{-1}(j))-\xi(\phi(j)|)^{2}\right)^{1/2}\\
&\leq \varepsilon'+\left(u_{d_{i}}(\{j:\sigma_{i}(e)(j)\ne j\})+2\varepsilon+\sum_{j:\sigma_{i}(e)(j)=j,|Z(\phi(j))|\leq M|}|w(\sigma_{i}(e)^{-1})-\xi(\phi(j))|^{2}\right)^{1/2}\\
&\leq \varepsilon'+\varespilon+\sqrt{2\varepsilon}+u_{d_{i}}(\{j:\sigma_{i}(e)(j)\ne j\})^{1/2}+\|w-G\circ \xi\circ \phi_{Z}\|_{2}\\
&\leq 2\varepsilon +(6+\sqrt{2})\sqrt{\varepsilon}+u_{d_{i}}(\{j:\sigma_{i}(e)(j)\ne j\})^{1/2}.
\end{align*}
Since $\varepsilon<1,$ we see that for all large $i,$ we have
\[\Delta_{Z,2}(\phi,\Gamma_{w,\psi_{Y}})<10\sqrt{\varepsilon}.\]
Thus
\[\pi_{Z}\circ \Map_{\mu}(\Delta_{X},F,\delta,L,\sigma_{i})\subseteq_{10\sqrt{\varepsilon}} \{\Gamma_{w,\psi_{Y}}:\psi\in D,\xi\in S_{\psi}\}.\]
Thus
\begin{align*}
N_{20\sqrt{\varepsilon}}(\pi_{Z}\circ \Map(\Delta_{X},F,\delta,L,\sigma_{i}),\Delta_{Z,2})&\leq N_{\varepsilon'/2}(\pi^{d_{i}}(\Map(\Delta_{Z}:\Delta_{X},F,\delta,L,\sigma_{i})),\Delta_{Y,\infty})\\
&\times M^{8\frac{\eta^{2}}{(1-\sqrt{\varepsilon})} d_{i}}\left(\frac{3+3\varepsilon}{\varpesilon}\right)^{\frac{8\eta^{2}}{(1-\sqrt{\varpesilon})} d_{i}},
\end{align*}
so
\begin{align*}
h_{(\sigma_{i}),\mu}(\Delta_{Z}:\Delta_{X},20\sqrt{\varepsilon},F,\delta,L)&\leq h_{(\sigma_{i})_{i},\mu}(\Delta_{Y}:\Delta_{X};\varepsilon'/2,F,\delta,L,\infty)+\frac{8\eta^{2}}{(1-\sqrt{\varespilon})^{2}}\log\left(\frac{3+3\varespilon}{\varepsilon}\right)\\
&+\frac{8\eta^{2}}{(1-\sqrt{\varepsilon})^{2}}\log M.
\end{align*}
Since this holds for all sufficiently large $F,\delta,L$ we can take the infimum over all $F,\delta,L$ to see that
\[h_{(\sigma_{i}),\mu}(\Delta_{Z}:\Delta_{X},40\sqrt{\varepsilon},\sigma_{i})\leq h_{(\sigma_{i})_{i},\mu}(\Delta_{Y}:\Delta_{X};\varepsilon'/2,\infty)+\frac{8\eta^{2}}{(1-\sqrt{\varespilon})^{2}}\log\left(\frac{3+3\varespilon}{\varepsilon}\right)+\frac{8\eta^{2}}{(1-\sqrt{\varepsilon})^{2}}\log M .\]
A fortiori,
\[h_{(\sigma_{i}),\mu}(\Delta_{Z}:\Delta_{X},40\sqrt{\varepsilon},\sigma_{i})\leq h_{(\sigma_{i})_{i},\mu}(Y:X,\Gamma)+\frac{8\eta^{2}}{(1-\sqrt{\varespilon})^{2}}\log\left(\frac{3+3\varespilon}{\varepsilon}\right)+\frac{8\eta^{2}}{(1-\sqrt{\varepsilon})^{2}}\log M.\]
Since $M$ only depends upon $\varepsilon,$ and $\eta$ can be any number less than $\varepsilon,$ we can let $\eta\to 0$ to see that
\[h_{(\sigma_{i}),\mu}(\Delta_{Z}:\Delta_{X},40\sqrt{\varepsilon},\sigma_{i})\leq h_{(\sigma_{i})_{i},\mu}(Y:X,\Gamma).\]
Taking the supremum over $\varepsilon$ completes the proof.

The ``in particular'' part follows since we may take $X=Z$ to see that
\[h_{(\sigma_{i})_{i},\zeta}(Z,\Gamma)=h_{(\sigma_{i})_{i},\zeta}(Z:Z,\Gamma)=h_{(\sigma_{i})_{i},\zeta}(Y:Z,\Gamma)\leq h_{(\sigma_{i})_{i},\nu}(Y,\Gamma).\]

\end{proof}

\section{Mixing and Strong Ergodicity over the Outer Pinsker Factor}

We give the definition of the \emph{Pinsker Factor} and the \emph{Outer Pinsker Factor}.

\begin{definition}\emph{Let $\Gamma$ be a countable discrete sofic group with sofic approximation $\sigma_{i}\colon\Gamma\to S_{d_{i}}.$ Let $(X,\mathcal{M},\mu)$ be a standard probability space and $\Gamma\actson (X,\mathcal{M},\mu)$ a measure-preserving action. We define $\Pi_{(\sigma_{i})_{i},\Gamma,X}$ to be the sigma-algebra generated by all finite observables $\alpha$ so that}
\[h_{(\sigma_{i})_{i},\mu}(\alpha:\mathcal{M},\Gamma)\leq 0.\]
\emph{Let $Y$ be the factor corresponding to $\Pi_{(\sigma_{i})_{i},\Gamma,X}$, it is easy to see that}
\[h_{(\sigma_{i})_{i},\mu}(Y:X,\Gamma)\leq 0.\]
\emph{And that $Y$ is the maximal factor of $X$ with $h_{(\sigma_{i})_{i},\mu}(Y:X,\Gamma)\leq 0.$ We will call $Y$ the} Outer Pinsker Factor \emph{for $\Gamma\actson (X,\mu)$. Similarly, we can show that there is a largest factor $\Gamma\actson (Y_{0},\nu_{0})$ which has entropy zero with respect to $(\sigma_{i})_{i}.$ We will call this the} Pinsker factor of $\Gamma\actson (X,\mu).$
\end{definition}

We have the following corollary of our first theorem.

\begin{cor}\label{C:emeddingstuff} Let $\Gamma$ be a countable discrete sofic group, and $\sigma_{i}\colon\Gamma\to S_{d_{i}}$ a sofic approximation. Let $(Y,\nu)$ be the Pinsker factor for $\Gamma\actson (X,\mu).$ Then  $L^{2}(X)\ominus  L^{2}(Y)$ embeds into $L^{2}(Y,\nu,\ell^{2}(\Gamma))^{\oplus \infty}$ as a representation of the $*$-algebra $L^{\infty}(Y)\rtimes_{\textnormal{alg}}\Gamma.$ Similarly, if $(Y_{0},\nu_{0})$ is the Outer Pinsker factor, then $L^{2}(X)\ominus L^{2}(Y_{0})$ embeds into $L^{2}(Y_{0},\ell^{2}(\Gamma))$ as a representation of $L^{\infty}(Y_{0})\rtimes_{\textnormal{alg}}\Gamma.$

\end{cor}

\begin{proof}
We do the proof only for the Pinsker factor, as the proof for the Outer Pinsker factor is the same. As in Proposition 4.3 of \cite{Me6} we may write
 \[L^{2}(X)\ominus L^{2}(Y)=\mathcal{H}_{a}\oplus \mathcal{H}_{s}\]
 where $\mathcal{H}_{a},\mathcal{H}_{s}$ are subrepresentation of the $*$-algebra $L^{\infty}(Y)\rtimes_{\textnormal{alg}}\Gamma,$ with $\mathcal{H}_{a}$ embedding into $L^{2}(Y,\nu,\ell^{2}(\Gamma))^{\oplus\infty}$ and $\mathcal{H}_{s}$ is singular with respect to $L^{2}(Y,\nu,\ell^{2}(\Gamma))$ as representations of $L^{\infty}(Y)\rtimes_{\textnormal{alg}}\Gamma.$ Let $\xi\in \mathcal{H}_{s}$ and let $\mathcal{A}$ be the smallest $\sigma$-algebra of measurable subsets of $X$ generated by $Y$ and $\xi.$ Let $(Z,\eta)$ be the factor of $(X,\mu)$ corresponding to $\mathcal{A}.$ Thus $(Z,\eta)$ is an intermediate factor of $(X,\mu)\to (Y,\nu).$ We know $\mathcal{K}=\overline{\Span\{g\xi:g\in\Gamma\}}^{\|\cdot\|_{2}}$ generates $Z$ over $Y.$ Tautologically, $\mathcal{K}$ is singular with respect to $L^{2}(Y,\nu,\ell^{2}(\Gamma))$ as a representation of $L^{\infty}(Y,\nu)\rtimes_{\textnormal{alg}}\Gamma.$ Thus by the preceding Theorem,
 \[h_{(\sigma_{i})_{i},\nu}(Z:X)=h_{(\sigma_{i})_{i},\nu}(Y:X)\leq 0.\]
 By maximality of the Outer Pinsker factor, we see that $Z=Y.$ This implies that $\xi\in L^{2}(Y),$ and thus that $\mathcal{H}_{s}=0.$

\end{proof}

We use the preceding  Corollary to show mixing and ergodicity properties of the factor $\pi\colon (X,\mu)\to (Y,\nu).$ We first recall some definitions. We use $\EE_{Y}(f)$ for the conditional expectation onto $Y$ of $f\in L^{1}(X,\mu).$ We will typically view $L^{2}(Y)$ inside of $L^{2}(X)$ by $f\mapsto f\circ \pi.$

\begin{definition}\emph{Let $\Gamma$ be a countable discrete group. Let $(X,\mu)$ be a standard probability space and $\Gamma\actson (X,\mu)$ a measure-preserving action. Let $\Gamma\actson (Y,\nu)$ be a factor of $(X,\mu)$ and $\alpha\colon \Gamma\to \mathcal{U}(L^{2}(X,\mu))$ be the Koopman representation. We say that the extension $\Gamma\actson (X,\mu)$ is mixing relative to $\Gamma\actson (Y,\nu)$ if for all $f,h\in L^{\infty}(X,\mu)$ with $\EE_{Y}(f)=0=\EE_{Y}(h)$ we have}
\[\lim_{g\to\infty}\|\EE_{Y}(\alpha(g)(f)h)\|_{L^{2}(Y)}=0.\]
\emph{We also that the extension $(X,\mu)\to (Y,\nu)$ is mixing. We say that the extension $(X,\mu)$ has spectral gap over $(Y,\nu)$ if for every sequence $\xi_{n}\in L^{2}(X)$ such that}
\[\lim_{n\to\infty}\|\alpha_{g}(\xi_{n})-\xi_{n}\|_{L^{2}(X)}=0 \mbox{\emph{ for all $g\in\Gamma$,}}\]
\emph{we have}
\[\lim_{n\to\infty}\|\xi_{n}-\EE_{Y}(\xi_{n})\|_{L^{2}(X)}=0.\]
\emph{We remark that is easy to see that $(X,\mu)$ has spectral gap over $(Y,\nu)$ if and only if $L^{2}(X)\ominus L^{2}(Y)$ has spectral gap as a representation of $\Gamma.$}

\end{definition}

For the proof we introduce some notation. Let $(Y,\nu)$ be as in the proceeding definition. For $f,h\in L^{\infty}(Y,\nu,\ell^{2}(\Gamma))$ we let $\ip{f,h}_{Y}\in L^{\infty}(Y,\nu)$ be defined by
\[\ip{f,h}_{Y}(y)=\sum_{g\in\Gamma}f(y)(g)\overline{h(y)(g)}.\]
For $f\in L^{\infty}(Y,\nu),\xi\in \ell^{2}(\Gamma)$ we let $f\otimes \xi\in L^{\infty}(X,\mu,\ell^{2}(\Gamma))$ be defined by
\[(f\otimes \xi)(y)(g)=f(y)\xi(g).\]
We let $\lambda_{Y}\colon\Gamma\to B(L^{\infty}(Y,\nu,\ell^{2}(\Gamma)))$ be defined by
\[(\lambda_{Y}(g)f)(y)(h)=f(g^{-1}y)(g^{-1}h).\]

\begin{theorem} Let $\Gamma$ be a countable discrete sofic group with sofic approximation $\sigma_{i}\colon\Gamma\to S_{d_{i}},$ and let $\Gamma\actson (X,\mu)$ be a probability measure-preserving action. Let $(Y,\nu),(Y_{0},\nu_{0})$ be the Pinsker and Outer Pinsker factors, respectively, of $\Gamma\actson (X,\mu).$

(i): The extension
\[\Gamma\actson (X,\mu)\to \Gamma\actson (Y,\nu)\]
is mixing. In particular,
\[\Gamma\actson (X,\mu)\to \Gamma\actson (Y_{0},\nu_{0})\]
is a mixing extension.

 (ii): If $\Lambda$ is any nonamenable subgroup of $\Gamma$,then
\[\Lambda\actson (X,\mu)\to \Lambda\actson (Y,\nu)\]
has spectral gap. In particular, the extension
\[\Lambda\actson (X,\mu)\to \Lambda\actson (Y_{0},\nu_{0})\]
has spectral gap.
\end{theorem}
\begin{proof}

Throughout, we shall let $\alpha_{X}\colon\Gamma\to \mathcal{U}(L^{2}(X)),$ $\alpha_{Y}\colon\Gamma\to \mathcal{U}(L^{2}(Y))$ be the Koopman representations.

(i): We shall use direct integral theory. See \cite{Taka}  IV.8 for the appropriate background. By the preceding Theorem, there is a $L^{\infty}(Y,\nu)\rtimes_{\textnormal{alg}}\Gamma$-equivariant, isometric, linear map
\[U\colon L^{2}(X)\ominus L^{2}(Y)\to L^{2}(Y,\mu,\ell^{2}(\Gamma))^{\oplus \infty}.\]
Let $\pi\colon (X,\mu)\to (Y,\nu)$ be the factor map. We may assume that $X,Y$ are standard Borel spaces and that $\pi$ is Borel. Let
\[(X,\mu)=\int_{Y}(X_{y},\mu_{y})\,d\nu(y)\]
be the disintegration. This means that $X_{y}=\pi^{-1}(\{y\}),$ that $\mu_{y}$ is a Borel probability measure supported on $X_{y},$ that $y\mapsto \mu_{y}(E)$ is measurable for all $E\subseteq X$ measurable and that
\[\int_{Y}\mu_{y}(E)\,d\nu(y)=\mu(E)\]
for all $E\subseteq X$ measurable. This allows us to regard
\[ L^{2}(X)\ominus L^{2}(Y)=\int_{Y}^{\oplus} L^{2}(X_{y},\mu_{y})\ominus \CC1\,d\nu(y),\]
\[L^{2}(X,\ell^{2}(\Gamma))=\int_{Y}^{\oplus}\ell^{2}(\Gamma)\,d\nu(y).\]
We may regard $\alpha_{X}(g)$ as a map $L^{2}(X_{y},\mu_{y})\to L^{2}(X_{gy},\mu_{gy}).$
Since $U$ is $L^{\infty}(Y)$-equivariant, we may argue as in \cite{Taka} Theorem IX.7.10 to see that there is a measurable field $U_{y}\in B(L^{2}(X_{y})\ominus \CC1,\ell^{2}(\Gamma))$ of isometric, linear maps so that
\[U=\int_{Y}^{\oplus}U_{y}\,d\nu(y).\]
By $\Gamma$-equivariance we have that $U_{gy}(\alpha_{X}(g)(\xi))=\lambda(g)U_{y}(\xi),$ for $\xi\in L^{2}(X_{y})\ominus \CC1.$
It is easy to see that for $f,h\in L^{2}(X)\ominus L^{2}(Y)$
\[\ip{\lambda_{Y}(g)U(f),U(h)}_{Y}=\EE_{Y}(\alpha_{X}(g)(f)h^{*}).\]
It thus suffices to show that
\[\lim_{g\to\infty}\|\ip{\lambda_{Y}(g)\xi,\eta}_{Y}\|_{2}=0\]
for $\xi,\eta\in L^{\infty}(X,\mu,\ell^{2}(\Gamma))$ with
\[\|\xi\|_{\infty},\|\eta\|_{\infty}\leq 1.\]
We first do this when
\[\xi=\sum_{s\in\Gamma}\xi_{s}\otimes\delta_{s}\]
\[\eta=\sum_{s\in\Gamma}\eta_{s}\otimes \delta_{s}\]
with all but finitely many terms in each sum equal to $0.$
It is then easy to see that
\[\ip{\lambda_{Y}(g)\xi,\eta}=0\]
for all $g$ outside a finite set. The case of a general $\xi,\eta$ follows by approximation as in Proposition 7.5 of \cite{Me6}.

The ``in particular'' follows as
\[\|\EE_{Y_{0}}((\xi\circ g^{-1})\eta)\|_{2}\leq \|E_{Y}((\xi\circ g^{-1})\eta)\|_{2}.\]

(ii): We consider the restriction of the action of $L^{\infty}(Y)\rtimes_{\textnormal{alg}}\Gamma$ to $\Gamma.$ By the preceding Corollary, this unitary representation of $\Gamma$ embeds into  $L^{2}(Y,\ell^{2}(\Gamma))^{\oplus\infty}.$
The unitary representation $L^{2}(Y,\ell^{2}(\Gamma))^{\oplus\infty}$ of $\Gamma$ is canonically isomorphic to $(\alpha_{Y}\otimes \lambda_{\Gamma})^{\oplus\infty}.$ By Fell's absorption principle, we know that  $\alpha_{Y}\otimes \lambda_{\Gamma}\leq \lambda_{\Gamma}^{\oplus\infty}.$ So the unitary representation $L^{2}(X)\ominus L^{2}(Y)$ of $\Gamma$ embeds into $\lambda_{\Gamma}^{\oplus\infty}.$ If we regard $L^{2}(X)\ominus L^{2}(Y)$ as a representation of $\Lambda$ then, by restriction, this representation embeds into $\lambda_{\Lambda}^{\oplus \infty}.$ By nonamenability of $\Lambda$ we have that $\lambda_{\Lambda}^{\oplus\infty}$ has spectral gap. Thus the extension of $\Lambda$-actions
\[\Lambda\actson (X,\mu)\to \Lambda\actson (Y,\nu)\]
has spectral gap. The ``in particular'' part follows from a similar analysis as in part $(i).$

\end{proof}

Now, suppose we are given a standard probability space $(X,\mathcal{M},\mu)$ and a countable discrete group $\Gamma$ with $\Gamma\actson (X,\mathcal{M},\mu)$ by measure-preserving transformations. Recall that the \emph{Furstenberg Tower} is a tower of complete, $\Gamma$-invariant subsigma-algebras $\mathcal{M}_{\alpha}$ of $\mathcal{M}$ indexed by ordinals $\alpha$ less than or equal to some countable ordinal $\lambda$ defined by:
\begin{enumerate}[(a):]
\item $\mathcal{M}_{0}\mbox{ is the sigma-algebra of sets which are either null or conull},$
\item if $\alpha$ is a successor ordinal and
\[(X_{\alpha},\mu_{\alpha})\to (X_{\alpha-1},\mu_{\alpha-1})\]
is the factor map corresponding to $\mathcal{M}_{\alpha-1}\subseteq \mathcal{M}_{\alpha},$ then $\mathcal{M}_{\alpha}$ is the largest sub-sigma-algebra of $\mathcal{M}$ with the property that the extension
\[(X_{\alpha},\mu_{\alpha})\to (X_{\alpha-1},\mu_{\alpha-1})\]
is compact,
\item if $\alpha$ is a limit-ordinal, then $\mathcal{M}_{\alpha}$ is the sigma-algebra generated by
\[\bigcup_{\alpha'<\alpha}M_{\alpha'},\]
\item if $(X_{\lambda},\mu_{\lambda})$ is the factor of $(X,\mu)$ corresponding to the sub-sigma-algebra $\mathcal{M}_{\lambda},$ then
\[(X,\mu)\to (X_{\lambda},\mu_{\lambda})\]
is a weakly mixing extension.
\end{enumerate}

Motivated by Theorem \ref{T:extensionnonsense}, we make the following definition.

\begin{definition}\emph{Let $(X,\mathcal{M},\mu)$ be a standard probability space and $\Gamma$ a countable discrete group with $\Gamma\actson (X,\mu)$ by automorphisms. Define, for every countable ordinal $\alpha,$ a family of complete, $\Gamma$-invariant, sub-sigma-algebras of $\mathcal{M}$ as follows:}
\begin{enumerate}[(i):]
\item $\mathcal{M}_{0}$\emph{ is the algebra of all null or conull sets}
\item \emph{if $\alpha$ is successor ordinal, and $(X_{\alpha-1},\mu_{\alpha-1})$ is the factor of $(X,\mu)$ corresponding to $\mathcal{M}_{\alpha-1},$ then we define $\mathcal{M}_{\alpha}$ to be the sigma-algebra generated by $\xi^{-1}(A),$ where $A$ is a Borel subset of $\CC,$ and $\xi\in L^{2}(X,\mu)$ has the property that}
\[\overline{\{\rho(f)\xi:f\in L^{\infty}(X_{\alpha-1},\mu_{\alpha-1})\rtimes_{\alg}\Gamma\}}^{\|\cdot\|_{2}}\]
\emph{is singular with respect to $\lambda_{X_{\alpha-1}}$ as a representation of $L^{\infty}(X_{\alpha-1},\mu_{\alpha-1})\rtimes_{\textnormal{alg}}\Gamma.$ }
\item \emph{if $\alpha$ is a limit ordinal, we let $\mathcal{M}_{\alpha}$ be the sigma-algebra generated by}
\[\bigcup_{\alpha'<\alpha}\mathcal{M}_{\alpha'}.\]
\end{enumerate}
\emph{We call $(\mathcal{M}_{\alpha})_{\alpha\leq \lambda}$ the} spectral tower.
\end{definition}

Note that if $(\mathcal{N}_{\alpha}),(\mathcal{M}_{\alpha})$ are the Furstenberg-Zimmer and spectral towers,respectively, then $\mathcal{N}_{\alpha}\subseteq \mathcal{M}_{\alpha}.$ Additionally, by iterated applications of Theorem \ref{T:extensionnonsense}, we find that if $(X_{\alpha},\mu_{\alpha})$ is the factor of $(X,\mu)$ corresponding to $\mathcal{M}_{\alpha}$ then
\[h_{(\sigma_{i})_{i},\mu_{\alpha}}(X_{\alpha},\Gamma)\leq 0.\]
 This is another way to see that distal measure-theoretic actions have nonpositive sofic entropy. Additionally, if $\lambda$ is the first ordinal for which $\mathcal{M}_{\lambda}=\mathcal{M}_{\lambda+1},$ and $(X_{\lambda},\mu_{\lambda})$ is the factor corresponding to $\mathcal{M}_{\lambda},$ then
\[L^{2}(X)\ominus L^{2}(X_{\lambda})\]
regarded as a representation of $L^{\infty}(X_{\lambda})\rtimes_{\textnormal{alg}}\Gamma$ embeds into $L^{2}(X_{\lambda},\ell^{2}(\Gamma))^{\oplus \infty}.$ This gives another perspective of the proof of Corollary \ref{C:emeddingstuff}.

It appears that the analogues of our results are not known for Rokhlin entropy. Thus we ask the following question.

\begin{question}\emph{The spectral tower makes sense for actions of arbitrary countable discrete groups. Is it true that if $(\mathcal{M}_{\alpha})_{\alpha}$ is the spectral tower of a probability measure-preserving action $\Gamma\actson (X,\mathcal{M},\mu)$ and if $\mathcal{M}_{\alpha}=\mathcal{M}$ for some $\alpha,$ then the Rokhlin entropy of $\Gamma\actson (X,\mu)$ is zero? Less ambitiously, is the Rokhlin entropy finite?}\end{question}

\appendix

\section{Extension Entropy in the Amenable Case}

The goal of this section is to prove that when $\Gamma$ is amenable group, then the extension entropy of
\[\Gamma\actson (X,\mu)\to \Gamma\actson (Y,\nu)\]
is just the entropy of $\Gamma\actson (Y,\nu).$ Though technical, our argument is fairly intuitive: we wish to show that every microstate for $\Gamma\actson (Y,\nu)$ ``lifts'' to one for $\Gamma\actson (X,\mu).$

The following Lemma is a formal version of this intuitive statement. It is essentially a direct translation of a result of Popa in \cite{PoppArg} (when the action is free the result is due to Pauenscu in \cite{LPaun}, but we will need the general version of Popa). We will use ultrafilters in our proof. If $\omega$ is a free ultrafilter on $\NN,$ then we say that $A\subseteq \NN$ is $\omega$-large if $A\in \omega.$ 

We also need the notion of a tracial ultraproduct of von Neumann algebras. Fix a free ultrafilter $\omega$ on $\NN.$ Suppose that $(M_{n},\tau_{n})_{n=1}^{\infty}$ is a sequence of tracial von Neumann algebras (i.e. $M_{n}$ are von Neumann algebras and $\tau_{n}$ is a faithful, normal, tracial state on $M_{n}$). Let
\[M=\prod_{n\to\omega}(M_{n},\tau_{n})=\frac{\{(x_{n})_{n=1}^{\infty}:x_{n}\in M_{n},\sup_{n}\|x_{n}\|<\infty\}}{\{(x_{n})_{n=1}^{\infty}:\sup_{n}\|x_{n}\|<\infty,\lim_{n\to\omega}\tau_{n}(x_{n}^{*}x_{n})=0\}}.\]
If $(x_{n})_{n=1}^{\infty}$ is a sequence with $x_{n}\in M_{n}$ and $\sup_{n}\|x_{n}\|<\infty,$ we use $(x_{n})_{n\to\omega}$ for the image of $(x_{n})_{n=1}^{\infty}$ in $M$ under the quotient map. There is a well-defined trace
\[\tau_{\omega}\colon \prod_{n\to\omega}(M_{n},\tau_{n})\to \CC\]
given by
\[\tau_{\omega}((x_{n})_{n\to\omega})=\lim_{n\to\omega}\tau_{n}(x_{n}).\]
We write
\[(M,\tau_{\omega})=\prod_{n\to\omega}(M_{n},\tau)\]
and say that $(M,\tau_{\omega})$ is the tracial ultraproduct of $(M_{n},\tau_{n})_{n=1}^{\infty}.$  Note that $M$ is a $*$-algebra with the operations being coordinate wise. Let $L^{2}\left(\prod_{n\to\omega}(M_{n},\tau_{n})\right)$ be the GNS completion of $M$ with respect to $\tau_{\omega}$ and let
\[\lambda\colon M\to B\left(L^{2}\left(\prod_{n\to\omega}(M_{n},\tau_{n})\right)\right)\]
be given by the GNS representation. It turns out that $\lambda(M)$ is a von Neumann algebra acting on $B\left(L^{2}\left(\prod_{n\to\omega}(M_{n},\tau_{n})\right)\right)$ (see \cite{BO} Lemma A.9). Thus we may regard $\prod_{n\to\omega}(M_{n},\tau_{n})$ as a tracial von Neumann algebra.

\begin{lemma}\label{L:approximatelift} Let $\Gamma$ be a countable, discrete, amenable group and $\sigma_{i}\colon\Gamma\to S_{d_{i}}$ a sofic approximation. Let $X,Y$ be compact, metrizable spaces with $\Gamma\actson X,\Gamma\actson Y$ by homeomorphisms and suppose that there is a factor map $\pi\colon X\to Y.$ Let $\Delta_{X},\Delta_{Y}$ be dynamically generating pseudometrics on $X,Y$ respectively. Let $\mu\in \Prob_{\Gamma}(X),\nu\in \Prob_{\Gamma}(Y)$ be such that $\pi_{*}\mu=\nu.$ Then for every finite $F\subseteq\Gamma,L\subseteq C(X)$ and $\delta,\varepsilon>0,$ there exists finite $F'\subseteq\Gamma,L'\subseteq C(Y)$ and $\delta'>0$ so that for all large $i$
\[\Map_{\nu}(\Delta_{Y},F',\delta',L',\sigma_{i})\subseteq_{\varepsilon,\Delta_{Y,2}} \pi\circ \Map_{\mu}(\Delta_{X},F,\delta,L,\sigma_{i}).\]

\end{lemma}

\begin{proof}

	Write
	\[\Gamma=\bigcup_{n=1}^{\infty}F_{n}',\]
	\[C(X)=\overline{\bigcup_{n=1}^{\infty}L_{n}'}\]
where $F_{n}'\subseteq \Gamma,L_{n}'\subseteq C(X)$ are increasing sequences of finite sets. Let $\delta_{n}'>0$ be a decreasing sequence converging to zero. Suppose that the Lemma is false, then we may find a finite $F\subseteq\Gamma,L\subseteq C(X)$ and $\delta,\varepsilon>0$ and a strictly increasing sequence of integers $i_{n}$ so that for all $n$ there is a $\phi_{n}\in \Map_{\nu}(\Delta_{X},F_{n}',\delta'_{n},L_{n},\sigma_{i_{n}})$ so that
\[\Delta_{Y,2}(\phi_{n},\pi\circ \psi)\geq \varespilon\mbox{ for all $\psi\in\Map_{\mu}(\Delta_{X},F,\delta,L,\sigma_{i_{n}})$}.\]
Since $\Gamma\actson (X,\mu)$ is amenable, we know by \cite{KLi2} that $h_{(\sigma_{i})_{i},\mu}(X,\Gamma)\geq 0$ and thus we can find $\psi_{n}\in X^{d_{i_{n}}}$ so that
\[\Delta_{X,2}(g\psi_{n},\psi_{n}\circ \sigma_{i_{n}}(g))\to_{n\to\infty}0,\]
\[(\psi_{n})_{*}(u_{d_{i_{n}}})\to \mu\mbox{ weak$^{*}$.}\]

	Fix a free ultrafilter $\omega\in\beta\NN\setminus\NN$ so that $\{i_{n}:n\in\NN\}\in\omega.$ Let
	\[(M,\tr_{\omega})=\prod_{n\to \omega}(M_{d_{i_{n}}}(\CC),\tr).\]
Define
\[\Phi\colon C(Y)\rtimes_{\textnormal{alg}}\Gamma\to M,\]
\[\Psi\colon C(X)\rtimes_{\textnormal{alg}}\Gamma\to M\]
by
\[\Phi(fu_{g})=(m_{f\circ \phi_{n}}\sigma_{i_{n}}(g))_{n\to \omega},\mbox{ for $f\in C(Y),g\in\Gamma$,}\]
\[\Psi(fu_{g})=(m_{f\circ \psi_{n}}\sigma_{i_{n}}(g))_{n\to\omega},\mbox{ for $f\in C(X),g\in\Gamma$,}\]
and extended by linearity. It is easy to see that $\Phi,\Psi$ are $*$-homomorphisms and that $\tr_{\omega}\circ \Phi=\tau_{\nu},\tr_{\omega}\circ \Psi=\tau_{\mu}.$ By uniqueness of GNS representations, we see that $\Phi,\Psi$ extend to normal $*$-homomorphisms
\[\overline{\Phi}\colon L^{\infty}(Y,\nu)\rtimes\Gamma\to M\]
\[\overline{\Psi}\colon L^{\infty}(X,\mu)\rtimes \Gamma\to M\]
so that $\tr_{\omega}\circ \overline{\Phi}=\tau_{\nu},\tr_{\omega}\circ \overline{\Psi}=\tau_{\mu}.$ We have an injective, trace-preserving, normal inclusion
\[\iota:L^{\infty}(Y,\nu)\rtimes \Gamma\to L^{\infty}(X,\mu)\rtimes\Gamma\]
defined densely by
\[\iota(f u_{g})=f\circ \pi u_{g}.\]
By Corollary 5.2 of \cite{PoppArg}, we may find a $p\in \prod_{n\to\omega}S_{d_{i_{n}}}$ so that
\[\overline{\Phi}(x)=p^{-1}\overline{\Psi}(\iota(x))p\mbox{ for all $x\in L^{\infty}(Y,\nu)\rtimes \Gamma$}\]
 (alternatively one can extend $\Psi$ to an embedding of $L^{\infty}(\{0,1\}^{\Gamma}\times X)\rtimes \Gamma$ using the arguments of Theorem 8.1 of \cite{Bow} and then use Proposition 1.20 of \cite{LPaun}). Write $p=(p_{n})_{n\to\omega}$ with $p_{n}\in S_{d_{i_{n}}}.$
The above equality then translates into
\[(m_{f\circ \phi_{n}})_{n\to\omega}=(m_{f\circ \pi\circ \psi_{n}\circ p_{n}})_{n\to\omega},\]
\[(p_{n}\sigma_{i_{n}}(g))_{n\to\omega}=(\sigma_{i_{n}}(g)p_{n})_{n\to\omega}.\]
From the above it is easy to see that
\begin{equation}\label{L:asycommutappendix}
\lim_{n\to\omega}u_{d_{i_{n}}}(\{j:p_{n}(\sigma_{i_{n}}(g)(j))=\sigma_{i_{n}}(p_{n}(j))\})=1.
\end{equation}

	We proceed to get a contradiction by establishing a few preliminary claims. We first claim that
\[\lim_{n\to\omega}\Delta_{Y,2}(\pi\circ \psi_{n}\circ p_{n},\phi_{n})=0.\]
To see this, let $\eta>0$ and let $M$ be the diameter of $\Delta_{Y}.$ Since $Y$ may be identified with the Gelfand spectrum of $C(Y),$ we see that we may find a $\kappa>0$ and $f_{1},\dots,f_{n}\in C(Y)$ so that if $y_{1},y_{2}\in Y$ have
\[\max_{1\leq r\leq n}|f_{j}(y_{1})-f_{j}(y_{2})|<\kappa,\]
then
\[\Delta_{Y}(y_{1},y_{2})<\eta.\]
Since for all $1\leq l\leq n$
\[\lim_{n\to\omega}\|f\circ \phi_{n}-f\circ \pi\circ \psi_{n}\circ p_{n}\|_{2}=0,\]
we see that  for all $1\leq l\leq n$
\[\lim_{n\to\omega}u_{d_{i_{n}}}(\{j:|f_{l}(\phi_{n}(j))-f_{l}(\pi\circ \psi_{n}\circ p_{n}(j))|<\kappa\})=1.\]
Thus, there is an $\omega$-large set of $n$ so that
\[u_{d_{i_{n}}}\left(\bigcap_{l=1}^{n}\{j:|f_{l}(\phi_{n}(j))-f_{l}(\pi\circ \psi_{n}\circ p_{n}(j))|<\kappa\}\right)\geq 1-\eta.\]
By our choice of $f_{1},\dots,f_{n},\kappa,$ we thus see that
\[\Delta_{Y,2}(\pi\circ \psi_{n}\circ p_{n},\phi_{n})^{2}<\eta^{2}+M^{2}(1-\eta)\]
for an $\omega$-large set of $n.$ Thus
\[\lim_{n\to\omega}\Delta_{Y,2}(\pi\circ \psi_{n}\circ p_{n},\phi_{n})\leq \eta^{2}+M^{2}(1-\eta)\]
and letting $\eta\to 0$ proves the first claim.

	The second claim is that
	\[ \psi_{n}\circ p_{n}\in \Map_{\nu}(\Delta_{Y,2},F,\delta,L,\sigma_{i}).\]
for an $\omega$-large set of $n.$ Suppose we grant this claim. The combination of the first and second claims imply that there is an $\omega$-large set of $n$ so that
\[\psi_{n}\circ p_{n}\in \Map_{\nu}(\Delta_{Y,2},F,\delta,L,\sigma_{i})\]
and
\[\Delta_{Y,2}(\phi_{n},\psi_{n}\circ p_{n})<\varepsilon,\]
which contradicts our choice of $\phi_{n}.$ It thus remains to prove the second claim. By standard properties of ultrafilters, to prove the second claim it is enough to show that
\begin{equation}\label{E:alsgaljappendix}
\lim_{n\to\omega}\Delta_{X,2}(\psi_{n}\circ p_{n}\circ \sigma_{i_{n}}(g),g\psi_{n}\circ p_{n})=0\mbox{ for all $g\in\Gamma,$}
\end{equation}
\begin{equation}\label{E:''alsdhgadklappendix}
\lim_{n\to\omega}(\psi_{n}\circ p_{n})_{*}(u_{d_{i_{n}}})=\mu \mbox{ weak$^{*}$.}
\end{equation}
Equation (\ref{E:''alsdhgadklappendix}) is a trivial consequence of the fact that $(\psi_{n})_{*}(u_{d_{i_{n}}})\to \mu$ in the weak$^{*}$ topology. Equation (\ref{E:alsgaljappendix}) is straightforward to argue from (\ref{L:asycommutappendix}) and the fact that
\[\Delta_{X,2}(\psi_{n}\circ \sigma_{i_{n}}(g),g\psi_{n})\to_{n\to\infty}0\mbox{ for all $g\in\Gamma.$}\]
This completes the proof.

\end{proof}

We can now prove that extension entropy agrees with entropy of the factor in the amenable case.

\begin{theorem} Let $\Gamma$ be a countable, discrete, amenable group and $\sigma_{i}\colon\Gamma\to S_{d_{i}}$ a sofic approximation. Let $(X,\mu),(Y,\nu)$ be standard probability spaces with $\Gamma\actson (X,\mu),\Gamma\actson (Y,\nu)$ by measure-preserving transformations and suppose that $\pi\colon X\to Y$ is a factor map. Then
\[h_{(\sigma_{i})_{i},\mu}(Y:X,\Gamma)=h_{\nu}(Y,\Gamma).\]
\end{theorem}

\begin{proof}
	By \cite{BowenAmen},\cite{KLi2} it is enough to show that
	\[h_{(\sigma_{i})_{i},\mu}(Y:X,\Gamma)=h_{(\sigma_{i})_{i},\nu}(Y,\Gamma).\]
It is clear that
\[h_{(\sigma_{i})_{i},\mu}(Y:X,\Gamma)\leq h_{(\sigma_{i})_{i},\nu}(Y,\Gamma),\]
 so we focus on proving the opposite inequality.

 	We may assume that $X,Y$ are compact, metrizable spaces, that the action of $\Gamma$ is by homeomorphisms, that $\mu\in\Prob_{\Gamma}(X),\nu\in\Prob_{\Gamma}(Y)$ and that $\pi$ is continuous. Fix dynamically generating pseudometrics $\Delta_{Y},\Delta_{X}$ on $Y,X.$  Let $\delta,\varespilon>0$ and finite $F\subseteq\Gamma,L\subseteq C(X)$ be given. Let $\delta',F',L'$ be as in the preceding Lemma. We then have that
 \[h_{(\sigma_{i})_{i},\nu}(\Delta_{Y},F',\delta',L',2\varepsilon)\leq h_{(\sigma_{i})_{i},\mu}(\Delta_{Y}:\Delta_{X},F,\delta,L,\varepsilon).\]
A fortiori,
\[h_{(\sigma_{i})_{i},\nu}(\Delta_{Y},2\varepsilon)\leq h_{(\sigma_{i})_{i},\mu}(\Delta_{Y}:\Delta_{X},F,\delta,L,\varepsilon)\]
and taking the infimum over all $F,\delta,L$ we have
\[h_{(\sigma_{i})_{i},\nu}(\Delta_{Y},2\varepsilon)\leq h_{(\sigma_{i})_{i},\mu}(\Delta_{Y}:\Delta_{X},\varepsilon).\]
Letting $\varepsilon\to 0$ completes the proof.

\end{proof}

\end{document}